\documentclass[10pt,reqno]{amsart}

\usepackage[a4paper,left=35mm,right=35mm,top=30mm,bottom=30mm,marginpar=25mm]{geometry}

\usepackage[colorlinks]{hyperref}

\usepackage{amsmath}
\usepackage{amssymb}
\usepackage{amsthm}

\usepackage[latin1]{inputenc}
\usepackage{eurosym}
\usepackage[dvips]{graphics}
\usepackage{graphicx}
\usepackage{epsfig}
\usepackage{hyperref}
\usepackage{dsfont}
\usepackage[nocompress, space]{cite}

\usepackage[displaymath,mathlines,pagewise]{lineno}
\allowdisplaybreaks

\usepackage{ifthen}

\renewcommand{\div}{\operatorname{div}}

\newcommand{\Rr}{{\mathbb{R}}}

\newcommand{\Nn}{{\mathbb{N}}}

\newcommand{\Tt}{{\mathbb{T}}}

\newcommand{\epsi}{\varepsilon}
\def\d{{\rm d}}
\def\dx{{\rm d}x}

\def\dt{{\rm d}t}
\def\leq{\leqslant}
\def\geq{\geqslant}

\numberwithin{equation}{section}

\newtheoremstyle{thmlemcorr}{10pt}{10pt}{\itshape}{}{\bfseries}{.}{10pt}{{\thmname{#1}\thmnumber{
                        #2}\thmnote{ (#3)}}}
\newtheoremstyle{thmlemcorr*}{10pt}{10pt}{\itshape}{}{\bfseries}{.}\newline{{\thmname{#1}\thmnumber{
                        #2}\thmnote{ (#3)}}}
\newtheoremstyle{defi}{10pt}{10pt}{\itshape}{}{\bfseries}{.}{10pt}{{\thmname{#1}\thmnumber{
                        #2}\thmnote{ (#3)}}}
\newtheoremstyle{remexample}{10pt}{10pt}{}{}{\bfseries}{.}{10pt}{{\thmname{#1}\thmnumber{
                        #2}\thmnote{ (#3)}}}
\newtheoremstyle{ass}{10pt}{10pt}{}{}{\bfseries}{.}{10pt}{{\thmname{#1}\thmnumber{
                        A#2}\thmnote{ (#3)}}}

\theoremstyle{thmlemcorr}
\newtheorem{theorem}{Theorem}
\numberwithin{theorem}{section}

\newtheorem{proposition}[theorem]{Proposition}

\theoremstyle{thmlemcorr*}
\newtheorem{theorem*}{Theorem}
\newtheorem{lemma*}[theorem]{Lemma}
\newtheorem{corollary*}[theorem]{Corollary}
\newtheorem{proposition*}[theorem]{Proposition}
\newtheorem{problem*}[theorem]{Problem}
\newtheorem{conjecture*}[theorem]{Conjecture}

\theoremstyle{defi}
\newtheorem{definition}[theorem]{Definition}
\newtheorem{hyp}{Assumption}

\newtheorem{problem}{Problem}

\theoremstyle{remexample}
\newtheorem{remark}[theorem]{Remark}
\newtheorem{example}[theorem]{Example}

\newtheorem{teo}[theorem]{Theorem}
\newtheorem{lem}[theorem]{Lemma}
\newtheorem{pro}[theorem]{Proposition}
\newtheorem{cor}[theorem]{Corollary}

\theoremstyle{ass}

\begin{document}
        
        \title[ Mean-Field Planning Problem]{A Potential  Approach for  Planning Mean-Field Games\\ in One Dimension}
        
\author{Tigran Bakaryan}
\address[T. Bakaryan]{
        King Abdullah University of Science and Technology (KAUST),
        CEMSE Division, Thuwal 23955-6900, Saudi Arabia.}
\email{tigran.bakaryan@kaust.edu.sa}
\author{Rita Ferreira}
\address[R. Ferreira]{
        King Abdullah University of Science and Technology (KAUST),
        CEMSE Division, Thuwal 23955-6900, Saudi Arabia.}
\email{rita.ferreira@kaust.edu.sa}
\author{Diogo Gomes}
\address[D. Gomes]{
        King Abdullah University of Science and Technology (KAUST),
        CEMSE Division, Thuwal 23955-6900, Saudi Arabia.}
\email{diogo.gomes@kaust.edu.sa}

         %Variational methods for elliptic systems
        
        \thanks{
                D. Gomes was partially supported by KAUST baseline and start-up funds and 
                KAUST SRI, Uncertainty Quantification Center in Computational Science and Engineering.  
                JD was partially supported by himself. 
        }
        \date{\today}

        \begin{abstract}
This manuscript discusses planning problems for first- and second-order one-dimensional mean-field games (MFGs). These games are comprised of a Hamilton--Jacobi equation coupled with  a Fokker--Planck equation. Applying  Poincar{\'e's} Lemma to the Fokker--Planck equation, we deduce the existence of a potential. Rewriting the  Hamilton--Jacobi equation in terms of the potential, we obtain a system of  Euler--Lagrange equations for certain  variational problems. Instead of the mean-field planning problem (MFP), we study this variational problem. By the direct method in the calculus of variations, we prove the existence and uniqueness of solutions to the variational problem. The variational approach has the advantage of eliminating  the continuity equation. 

We also consider a first-order MFP with congestion. We prove that the congestion problem has a weak solution by introducing a potential and relying on the theory of variational inequalities.  We end the paper by presenting an application to the one-dimensional Hughes' model.
        \end{abstract}
        
        \maketitle
%        \tableofcontents
        
        \section{Introduction}
        Mean-field games (MFGs) describe the interaction between a large number of identical rational agents. More precisely, every agent minimizes the same value function. MFGs first appeared in  \cite{ll1,ll2,ll3}  and, independently, in  \cite{huang2006large}. An important problem in the theory of MFGs is the mean-field planning (MFP) problem introduced in \cite{cursolionsplanning}, where  initial and final distributions for the population are predetermined, but agents are free to choose their strategies at intermediate
times. This problem is modeled by a  system of PDEs, comprised
by  a Hamilton--Jacobi equation and a Fokker--Planck equation,
and endowed with initial-terminal boundary conditions for the density of the agents. A typical example is
        \begin{equation}
        \label{MFGs}
        \begin{cases} 
        -u_t-\varepsilon \Delta u+H(x,Du)=g(m)\quad &\text{in}\,\ (0,T)\times \Tt^d  \\ m_t-\varepsilon \Delta m-\div (mH_p(x,Du))=0 \quad &\text{in}\,\ (0,T)\times \Tt^d \\
        m(0,x)=m_0(x), \quad m(T,x)=m_T(x) &\text{in}\,\ \Tt^d,
        \end{cases} 
        \end{equation}
        where $m$ is a probability density  and represents the agents' distribution in the space, and $u$ is the value function.
Note that because no boundary values are prescribed for \(u\), the unknown \(u\) can only be determined up to additive constants.

        The existence of classical solutions for the planning problem was considered in \cite{cursolionsplanning}, where the authors establish  the existence and uniqueness of classical solutions in the case in which the Hamiltonian is quadratic, $H(x, p)=\frac{|p|^2}{2}$, and  $g=g(m)$  is an increasing function.  A priori estimates for classical solutions of the planning problem were discussed in \cite{gomes2018displacement,BakaryanFerreiraGomes2020}, while the   existence and uniqueness of weak solutions for a wide range of Hamiltonians were addressed in  \cite{porretta,porretta2,OrPoSa2018,Tono2019}. In \cite{CDY} (see also \cite{achdou2013finite}),  the authors developed an efficient numerical method for the mean-field planning problem using
a variational approach. Here, we investigate an alternative variational approach for the planning problem. 
%One possible application of
%this approach is the construction of numerical methods. 
As we will see, our technique relies on introducing a potential that
integrates the second equation  in \eqref{MFGs}, reducing the complexity of the problem.

        The MFG in \eqref{MFGs} can be derived from an optimal control problem;  as shown in  \cite{ll2}, the first equation in \eqref{MFGs}  is the Euler--Lagrange equation for the following optimal control problem
        \begin{equation*}
        \inf\limits_{v\in \mathcal{A}} \int_{0}^{T}\!\!\!\int_{\Tt^d}^{} \big( L(x,v)m+G(m)\big) \,\dx\dt
        \end{equation*}
        subject to
        \begin{equation}\label{constraint}
        \begin{cases} 
        m_t-\varepsilon \Delta m-\div (mv)=0 &\quad \text{in}\,\ (0,T)\times \Tt^d\\
        m(t,x)\geq 0& \quad \text{in}\,\ (0,T)\times \Tt^d\\
        m(0,x)=m_0(x)& \quad\text{in}\,\ \Tt^d\\
        m(T,x)=m_T(x)& \quad\text{in}\,\ \Tt^d,
        \end{cases} 
        \end{equation}
        where $\mathcal{A}$ is a set of smooth enough vector fields, $v:[0,T]\times \Tt^d\to\Rr^d$.
        Here, the Lagrangian $L:\Tt^d\times \Rr^d\to\Rr$ is the Fenchel
        conjugate of the Hamiltonian $H$ with respect to the second variable and \(G'=g\). This variational  formulation was used in \cite{CDY} to develop a numerical method for the mean-field planning problem. In this setting, the  constraints in \eqref{constraint} contain the continuity equation (the first equation in \eqref{constraint}). Another variational approach for the mean-field planning problem was considered in \cite{OrPoSa2018,Tono2019}, where the authors define  weak solutions to problem \eqref{MFGs} through the solutions to a dual variational problem. The set of constraints of that dual problem  also contains the continuity equation (for more details, see \cite{OrPoSa2018} or  \cite{Tono2019}).
        Thus, in both cases, the set of constraints involves the continuity equation. In this paper, for the one-dimensional case in space, we introduce an alternative  variational approach for the mean-field planning problem whose set of constraints does not contain the continuity equation. As we mentioned before,
one of the main advantages of the absence of the continuity equation constraint   that it may enable more straightforward numerical approaches.  
       
      In the first part of this paper, we consider the following first- and  second-order one-dimensional mean-field planning problems.
        \begin{problem} \label{planningSG} Suppose that $m_0$, $m_T\in \mathcal{P(\Tt)} \cap C^2(\Tt)$, \( H \in C^2(\Rr)\), \( V \in C^1(\Tt)\),  and $g\in C^1(\Rr_0^+)$. Assume further that $H$ is strictly convex, $g$ is  non-decreasing,  and $\lambda=0$ or $\lambda=1$. Find  $(u,m)\in C^2([0,T]\times \Tt)\times C^2([0,T]\times \Tt)$ satisfying  $m\geq 0$,
                \begin{equation}\label{planningSGS}
                \begin{cases} 
                -u_t-\lambda u_{xx}+H(u_x)+V(x)=g(m) &  \\ m_t-\lambda m_{xx}- (H^\prime(u_x) m)_x=0 & 
                \end{cases} \text{in}\,\ (0,T)\times \Tt,
                \end{equation}
                and
                \begin{equation}
                \label{boundaryP}
                \begin{cases} 
                m(0,x)=m_0(x) &  \\ m(T,x)=m_T(x)& 
                \end{cases} \text{in}\,\ \Tt.
                \end{equation} 
        \end{problem}  
        The first-order mean-field planning problem corresponds to  $\lambda =0$, while  the second-order one corresponds  to $\lambda=1$. 
        
        In this study, we address  Problem \ref{planningSG} by deriving new associated variational problems, both for \(\lambda=0\) and \(\lambda=1\). Then, under suitable conditions on the data of Problem \ref{planningSG}, we prove the existence and uniqueness of  solutions to those  variational problems. In turn, if the solutions to the variational
problems are regular enough, we recover solutions to Problem \ref{planningSG}.
        
        %         Referring to this, we define  weak solutions of  Problem \ref{planningSG} with \(\lambda=0\) and \(\lambda=1\) as solutions of corresponding variantional problems.
        %        
        To present our approach, let $L:\Rr\to\Rr$ be the  Legendre transform of $H$,
        \begin{equation*}
L(w)=\sup\limits_{p\in\Rr}\big(pw-H(p)\big),\quad w\in\Rr,
        \end{equation*}
and let  $G:\Rr\to\Rr$ and $L_0:\Rr\times\Rr_0^+\to\Rr$ be such that
        \begin{equation}\label{G-def}
        G^\prime(z)=g(z), \quad z\in\Rr^+_0,
        \end{equation}
        and  
        \begin{equation}
        \label{defLTilde}
        L_0(z,y)=
        \begin{cases} 
        L\big(  \tfrac{z}{y}\big) y,  &(z,y)\in\Rr\times\Rr^+,   \\ +\infty, \,\, &z\neq 0,\, y=0, 
        \\ 0, \,\, &z=0,\,y= 0.
        \end{cases}
        \end{equation}
Because  \( H \in C^2(\Rr)\) and it is strictly convex,  \(L\) is a  convex function with 
        \begin{equation}\label{L}
 L^\prime=\left( H^\prime\right)^{-1}.
        \end{equation}
     
        As we detail in Section~\ref{sec1}, the next problem outlines the variational problem associated with our mean-field planning problem, Problem \ref{planningSG}.  
        \begin{problem}\label{problem}Consider the setting of Problem \ref{planningSG}.  Let $p>1$, 
                \begin{equation}\label{B00}
        \begin{aligned}
        \mathcal{B}_{0}^{p}=\bigg\{&\varphi\in W^{1,p}([0,T]\times\Tt)\!:\varphi_x+1\geq 0, \,\int_{\Tt}\varphi(t,x)\,\dx=0, \\  
&\qquad  \varphi(0,x)=\int_{0}^{x}\big(m_0(\tau)-1\big)\,\d \tau - \int_{\Tt}\int_{0}^{x}(m_0(\tau)-1)\,\d \tau\dx,\\ 
&\qquad \varphi(T,x)=\int_{0}^{x}\big(m_T(\tau)-1\big)\,\d \tau - \int_{\Tt}\int_{0}^{x}(m_T(\tau)-1)\,\d \tau\dx \bigg\}\times L^p([0,T]),
        \end{aligned}
        \end{equation}  
        and 
                \begin{equation}\label{B001}
        \begin{split}
        \mathcal{B}_{1}^{p}=\big \{(\varphi,q)\in\mathcal{B}_{0}^{p}\!:\varphi_{xx} \in L^p([0,T]\times\Tt)\big\}.
        \end{split}
        \end{equation}
                 Find  $(\varphi,q)\in \mathcal{B}_{\lambda}^{p}$ that minimizes the  functional 
                \begin{equation}\label{prof}
                %                   \label{problemfunc}
                \int_{0}^{T}\!\!\!\int_{\mathbb{T}}\big (L_0(\varphi_t+q-\lambda\varphi_{xx},\varphi_x+1)-V \varphi_x+G( \varphi_x+1)\big)\,\dx\dt
                \end{equation}
                over \(W^{1,p}([0,T]\times\Tt)\times L^p([0,T]) \), where $G$ and \(L_0\) are given by  \eqref{G-def}--\eqref{defLTilde}.
                \end{problem}
         Problem~\ref{problem} encodes  two different  variational problems according to the value of \(\lambda\), \(\lambda=0\) or \(\lambda=1\). When $\lambda=0$, the variational problem  in Problem~\ref{problem} is associated with the first-order mean-field planning problem corresponding to  $\lambda=0$ in Problem~\ref{planningSG}; when \(\lambda=1\), the variational problem  in Problem~\ref{problem} is  associated with the second-order mean-field planning problem corresponding to $\lambda=1$ in Problem~\ref{planningSG}.  In Section~\ref{sec1}, we  present a formal derivation of  Problem~\ref{problem}  and describe the  variational problem corresponding to the first-order MFP in higher dimensions, $d>1$.
         
The next proposition states sufficient conditions under which solutions to  Problem~\ref{problem} provide solutions to Problem~\ref{planningSG}.
                \begin{pro}\label{prorel} Let \(\lambda\in \{0,1\} \). Suppose that $(\varphi,q)\in C^{\lambda+2}([0,T]\times \Tt)\times C^{\lambda+1}([0,T])$ solves Problem~\ref{problem}. Assume further that $\varphi_x+1>0$ and, for \((x,t)\in [0,T]\times \Tt\), set
        \begin{equation}\label{proreleq}
        \begin{cases}
        m(t,x)=\varphi_x(t,x)+1   \\  u(t,x)=\int_{0}^{x}L^\prime\big(\frac{ \varphi_t(t,\tau)+q(t)-\lambda \varphi_{xx}(t,\tau)}{\varphi_x(t,\tau)+1}\big)\,\d \tau .
        \end{cases}
        \end{equation}
         Then, there exits a function, \(\vartheta:[0,T]\to\Rr\), only depending on \(t\), such that $(\tilde u, m)=(u + \vartheta,m)$ solves Problem~\ref{planningSG}.
        \end{pro}

        In Section~\ref{assumptions}, we state the assumptions used in our analysis.  In Section~\ref{preli}, we introduce additional notation and prove some preliminary results. With these results, we establish the following theorem.
        \begin{teo}\label{exist2} Suppose that Assumptions~\ref{assumtionOnG2}--\ref{assumtionOnH2}
hold  for some \(\beta>1\) and \(\gamma>1\) (see Section~\ref{assumptions}).
Then, Problem~\ref{problem} has a unique solution; more precisely, the variational problems corresponding to  $\lambda=0$ and $\lambda=1$ in Problem~\ref{problem} have unique solutions in  $ \mathcal{B}_0^ \frac{ \beta\gamma}{\beta+\gamma-1}$ and  $\mathcal{B}_1^ \frac{ \beta\gamma}{\beta+\gamma-1}$, respectively. 
        \end{teo}
    In Sections~\ref{sec2} and \ref{sec3}, we address    Problem~\ref{problem}. In particular, in Section~\ref{sec2}, we derive the variational problem corresponding to  $\lambda=0$ in Problem~\ref{problem} and prove the existence and uniqueness of the minimizer. In Section~\ref{sec3}, we treat the  $\lambda=1$ case. 
    
The second part of this paper is devoted to studying the following first-order planning problem with congestion.
        \begin{problem} \label{con} Suppose that $m_0$, $m_T\in \mathcal{P}(\Tt) \cap C^2(\Tt)$ and $\alpha$, $\mu>0$. Find  $(u,m)\in C^2([0,T]\times \Tt)\times C^2([0,T]\times \Tt)$ satisfying  $m\geq 0$,
                \begin{equation*}
                \begin{cases} 
                -u_t+\frac{u_x^2}{2 m^\alpha}=m^\mu &  \\ m_t- (u_x m^{1-\alpha})_x=0
                \end{cases} \text{in}\,\ (0,T)\times \Tt,
                \end{equation*}
                and the boundary conditions in \eqref{boundaryP}.
        \end{problem}  
        To study Problem \ref{con}, we proceed as before and introduce a potential (see Section \ref{sec1} for more details). In contrast with the methodology used to address Problem \ref{planningSG}, in this case,  the  potential method {\color{red}{leads}} us {\color{red}to} a system of PDE's that is not an Euler--Lagrange equation. In particular, by using the method of potential, we derive the following problem.
        
         \begin{problem}\label{con-var}Consider the setting of Problem \ref{con} and  let $p>1$. Find $(\varphi,q)\in \mathcal{B}_0^{p}$ satisfying
                \begin{equation*}
        \begin{cases}
        - \Big( \frac{\varphi_t+q}{ (\varphi_x+1)^{1-\alpha}}\Big)_t  +\frac{1}{2}\Big(\frac{(\varphi_t+q)^2}{ (\varphi_x+1)^{2-\alpha}} \Big)_x- ((\varphi_x+1)^\mu)_x=0\quad  \text{in}\,\ (0,T)\times \Tt\\
\int_{\mathbb{T}}(\varphi_x+1)^{\alpha-1}(\varphi_t+q)\,\dx=0\quad \text{in}\,\ (0,T)
        \end{cases},
                \end{equation*}  
                where  $\mathcal{B}_0^{p}$ is defined by \eqref{B00}.
        \end{problem}
       We solve this system using variational inequalities methods; see, for example,  \cite{FG2} and \cite{FGT1}. 
The next theorem provides the existence of solutions to Problem \ref{con-var}.
\begin{teo}\label{con-ex} Let \(\alpha\in(0,2)\) and \(\mu>0\) be such that $\alpha<\mu+1$. Suppose that Assumption~\ref{assumtionOnBoundsA}
holds with    \(m_0\),
\(m_T \in H^{5}(\Tt)\).  Then, there exists a weak solution $(\varphi,q)\in \mathcal{B}_0^{\kappa}$ to Problem~\ref{con-var} (in the sense of Definition~\ref{def-weak}),  where $\kappa=\min\big\{\frac{2(\mu+1)}{\mu+2-\alpha},\mu+1\big\}$ when $\alpha\in(0,1]$ and $\kappa=\mu+1$ when $\alpha\in(1,2)$.
\end{teo}
Similarly to Proposition \ref{prorel}, the following  proposition states sufficient conditions under which the weak solutions to  Problem~\ref{con-var} provide solutions to Problem~\ref{con}.
\begin{pro}\label{prorelc}  Let  $(\varphi,q)\in C^{2}([0,T]\times \Tt)\times C^{1}([0,T])$ solve Problem \ref{con-var} (in the sense of Definition~\ref{def-weak}). Assume that $\varphi_x+1>0$ and
        \begin{equation*}
        \begin{cases}
        m(t,x)=\varphi_x(t,x)+1   \\  u(t,x)=\int_{0}^{x}(\varphi_x+1)^{\alpha-1}(\varphi_t+q)\,\d \tau .
        \end{cases}
        \end{equation*}
Then, there exists a function, \(\vartheta:[0,T]\to\Rr\), only depending on \(t\), such that $(\tilde u, m)=(u + \vartheta,m)$ solves Problem~\ref{con}.
\end{pro}

We conclude this paper by presenting another application of the potential method in Section \ref{sec4}. Using this method, we explicitly solve the one-dimensional Hughes' model system in \eqref{Hu1}.
        
        \section{Motivation For The Variational Problem}
        \label{sec1}
        
              Here, we formally derive the variational  problem in Problem~\ref{problem}  and discuss the variational approach corresponding to the first-order MFP in $d>1$ dimension. As we mentioned in the Introduction, our
technique relies on introducing a potential function, \(\varphi\),
that integrates the second equation in \eqref{MFGs}. We start
by providing an interpretation of this potential function. Then, we give the formal derivation of the variational problem. We end this section by describing the  variational approach in $d>1$ dimension.
  
\
\paragraph{\textbf{Interpretation of the potential}} Let $\mu\in
\mathcal{P}(\Rr) \cap C^1([0,T]\times\Rr)$  and $J\in C^1([0,T]\times\Rr)$.
Suppose $(\mu,J)$ satisfies the following transport equation
\begin{equation}\label{transport}
\div_{(t,x)}(\mu,-J) = \mu_t- J_x=0.
\end{equation}

By Poincar{\'e's} lemma (see \cite[Theorem 1.22]{Csato2011ThePE}),
there exists a function $\psi:(0,T)\times \Rr\to\Rr$ such that
        \begin{equation}\label{intrip}
\begin{cases}
\mu=\psi_x\\
J=\psi_t.
\end{cases}
\end{equation}  
The first equation in \eqref{intrip} implies that $\psi$ is the
cumulative distribution of $\mu$; that is, for  any level set
$\psi(t,x(t))\equiv a$, $a\in \Rr^+_0$, we have 
\begin{equation*}
\int_{-\infty}^{x(t)}\mu(t,\tau)\,\d \tau=a.
\end{equation*}
Hence, the level sets of $\psi$ can be interpreted as quantiles.
Furthermore, because $\psi_t=J$, the quantile at a point $(t,x)$
increases according to the current, $J$, through that point as
can be seen by the identity
\begin{equation*}
\psi_t=\frac{\d}{\dt}\int_{-\infty}^{x}\mu(t,\tau)\,\d \tau=J(t,x).
\end{equation*}

We further observe that the equation in \eqref{transport} can
be rewritten as 
\begin{equation*}
\mu_t+ \left( \mu\big( -\tfrac{J}{\mu}\big) \right) _x=0,
\end{equation*} 
where $\big( -\tfrac{J}{\mu}\big)$ is the velocity field. Consider the
trajectory corresponding to the velocity $\big( -\tfrac{J}{\mu}\big)$;
that is, 
\begin{equation}\label{velo}
\dot x=-\frac{J(t, x(t))}{\mu(t, x(t))}.
\end{equation}
From \eqref{intrip} and \eqref{velo}, we have
\begin{equation*}\
\frac{\d}{\dt}\psi(t,x(t))=\psi_x\dot x+\psi_t=0.
\end{equation*}
Thus, following the velocity field, the quantile is constant
in time.

\
\paragraph{\textbf{Derivation of the variational problem}}        
        We now formally derive the variational  problem
in Problem~\ref{problem}. To
show the connection between this problem and Problem~\ref{planningSG},
we suppose that all functions considered in the remaining of this section are
smooth enough. Later, we discuss the precise functional setting
needed to make our results rigorous.

Suppose that $(u,m)$ solves Problem~\ref{planningSG} with $m>0$. Then, the second equation in \eqref{planningSGS} can be written as
        \begin{equation}        \label{div1}
        \div_{(t,x)}(m,-\lambda m_x-H^\prime(u_x)m)=0.
        \end{equation}
        %        which implies 
        %        \begin{equation}
        %        \div_{(t,x)}(m+c,-\lambda m_x-u_xm-q)=0
        %        \end{equation}
        %        for any constants $c,q\in\Rr$. 
        As before, by Poincar{\'e}'s lemma (see \cite[Theorem 1.22]{Csato2011ThePE}) and  \eqref{div1}, there exist functions, $\varphi:(0,T)\times \Tt\to\Rr$ and $q:(0,T)\to\Rr$,  and a constant $c\in \Rr$ such that 
        \begin{equation}\label{um_cq}
        \begin{cases}
        m=\varphi_x+c   \\ \lambda m_x +H^\prime(u_x)m=\varphi_t+q(t).
        \end{cases}
        \end{equation}
By adding a constant to $\varphi$, we can assume that $\int_{0}^{T}\!\!\!\int_{\mathbb{T}}\varphi\,\dx\dt=0$. We observe further that the choice of $q$ is not unique. Here,
  we select $q$ such that $\int_{\mathbb{T}}\varphi\,\dx=0$ for all $0\leq t\leq T$. Because $m$ is a probability density, we have $c=1$ and $m=\varphi_x+1>0$.  Therefore, using \eqref{L} in \eqref{um_cq}, we get
        \begin{equation}\label{um_pi}
        \begin{cases}
        m=\varphi_x+1   \\  u_x=L^\prime\big(\frac{ \varphi_t+q-\lambda \varphi_{xx}}{\varphi_x+1}\big).
        \end{cases}
        \end{equation}
         
We use the transformation in \eqref{um_pi} to deduce the variational formulation in Problem~\ref{problem}.         
        First,  we notice that the second equation in \eqref{um_pi} implies
        \begin{equation}\label{int_ux}
        0=\int_{\mathbb{T}} u_x\,\dx =\int_{\mathbb{T}} L^\prime\Big(\frac{ \varphi_t+q-\lambda \varphi_{xx}}{\varphi_x+1}\Big)\,\dx. 
        \end{equation}
        On the other hand, recalling  that $\varphi_x+1>0$,   we get from \eqref{defLTilde} that
        \begin{equation}\label{q11}
        \frac{\partial}{\partial q}L_0(\varphi_t+q-\lambda\varphi_{xx},\varphi_x+1)=
        L^\prime\Big(\frac{ \varphi_t+q-\lambda \varphi_{xx}}{\varphi_x+1}\Big).
        \end{equation}
        The two preceding identities,  \eqref{int_ux} and \eqref{q11}, imply  that
        \begin{equation}\label{qEq}
        \begin{split}
        \int_{\mathbb{T}} \frac{\partial}{\partial q}L_0(\varphi_t+q-\lambda\varphi_{xx},\varphi_x+1)\,\dx =0.
        \end{split}
        \end{equation}

 Next, we write the first equation in \eqref{planningSGS} in terms of $\varphi$ and $q$. After differentiating with respect to $x$, the first equation in \eqref{planningSGS} becomes
        \begin{equation}\label{HJux}
        -\left( u_{x}\right)_t -\lambda \left( u_{x}\right)_{xx} +\big( H(u_x)\big)_x+V^\prime =m_xg^\prime(m).
        \end{equation}
        Accordingly, using \eqref{um_pi}, we  get
        \begin{equation}\label{phiEq0}
        \begin{split}
        &-\left( L^\prime\Big(\frac{ \varphi_t+q-\lambda \varphi_{xx}}{\varphi_x+1}\Big)\right)_t-\lambda \left(L^\prime\Big(\frac{ \varphi_t+q-\lambda \varphi_{xx}}{\varphi_x+1}\Big)\right)_{xx}  \\&\quad+ \left( H \Big( L^\prime\Big(\frac{ \varphi_t+q-\lambda \varphi_{xx}}{\varphi_x+1}\Big)\Big) \right)_x+V^\prime-(\varphi_x+1)_xg^\prime(\varphi_x+1)=0.
        \end{split}
        \end{equation}
On the other hand,         the  identity in \eqref{L} implies that
        \begin{equation*}
\begin{split}
        \left( H \Big( L^\prime\Big(\frac{ \varphi_t+q-\lambda \varphi_{xx}}{\varphi_x+1}\Big)\Big) \right)_x&= H^\prime \Big( L^\prime\Big(\frac{ \varphi_t+q-\lambda \varphi_{xx}}{\varphi_x+1}\Big)\Big)\left( L^\prime\Big(\frac{ \varphi_t+q-\lambda \varphi_{xx}}{\varphi_x+1}\Big)\right)_x\\&=\frac{ \varphi_t+q-\lambda \varphi_{xx}}{\varphi_x+1}\left(L^\prime\Big(\frac{ \varphi_t+q-\lambda \varphi_{xx}}{\varphi_x+1}\Big) \right)_x.
\end{split}
        \end{equation*}
        This identity and \eqref{phiEq0} yield
        \begin{equation}\label{phiEq3}
        \begin{split}
        &-\left( L^\prime\left(\frac{ \varphi_t+q-\lambda \varphi_{xx}}{\varphi_x+1}\right)\right)_t-\lambda \left(L^\prime\Big(\frac{ \varphi_t+q-\lambda \varphi_{xx}}{\varphi_x+1}\Big)\right)_{xx}  \\&\quad+ \frac{ \varphi_t+q-\lambda \varphi_{xx}}{\varphi_x+1}\left(L^\prime\Big(\frac{ \varphi_t+q-\lambda \varphi_{xx}}{\varphi_x+1}\Big) \right)_x+V^\prime-\big(\varphi_x+1\big)_x\,g^\prime(\varphi_x+1)=0.
        \end{split}
        \end{equation}

        \begin{pro}\label{El1} Let $\lambda\in\{0,1\}$, $\mathcal{A}=C^{\lambda+2}([0,T]\times \Tt)\times C^{\lambda+1}([0,T])$, $G$ be given by \eqref{G-def} and $L_0$ by \eqref{defLTilde}. Suppose that $(\varphi,q)$ minimizes the functional 
                \begin{equation}\label{varProL0}
                \int_{0}^{T}\!\!\!\int_{\mathbb{T}} \big(L_0(\varphi_t+q-\lambda\varphi_{xx},\varphi_x+1)-V\varphi_{x}+G( \varphi_x+1)\big)\,\dx\dt
                \end{equation}
                over $\mathcal{A}$ and satisfies the condition $\varphi_x+1>0$. Then, the  Euler--Lagrange equations corresponding to the functional in \eqref{varProL0} are comprised by \eqref{qEq} and \eqref{phiEq3}.
        \end{pro}
        \begin{proof} Because  $q=q(t)$ depends only on  $t$, the Euler--Lagrange equation of the functional in \eqref{varProL0} with respect to $q$  is
                \begin{equation*}
        \begin{split}
        \frac{\partial}{\partial q}\int_{\mathbb{T}}\big(L_0(\varphi_t+q-\lambda\varphi_{xx},\varphi_x+1)-V\varphi_{x}+G( \varphi_x+1)\big)\,\dx =0.
        \end{split}
        \end{equation*}
        Exchanging the derivative with the integral, we get \eqref{qEq}.

        The Euler--Lagrange equation of the functional in \eqref{varProL0} with respect to $\varphi$ is 
                \begin{equation}\label{el1}
                \begin{aligned}
                &-\left( \frac{\partial}{\partial \varphi_{x}}  L_0(\varphi_t+q-\lambda\varphi_{xx},\varphi_x+1)-V+ g(\varphi_x+1) \right)_x\\
                &\quad-\left( \frac{\partial}{\partial \varphi_{t}}  L_0(\varphi_t+q-\lambda\varphi_{xx},\varphi_x+1) \right)_t+\left( \frac{\partial}{\partial \varphi_{xx}}  L_0(\varphi_t+q-\lambda\varphi_{xx},\varphi_x+1) \right)_{xx}=0.
                \end{aligned}
                \end{equation}
Expanding the first term in \eqref{el1}, we get
\begin{equation}\label{phiEq}
\begin{aligned}
&-\left(\frac{\partial}{\partial \varphi_{x}}L_0(\varphi_t+q-\lambda\varphi_{xx},\varphi_x+1) \right)_x+V^\prime-\big(\varphi_x+1\big)_x\,g^\prime(\varphi_x+1) \\
&\quad-\left( \frac{\partial}{\partial \varphi_{t}}L_0(\varphi_t+q-\lambda\varphi_{xx},\varphi_x+1)\right)_t+ \left(\frac{\partial}{\partial \varphi_{xx}}L_0(\varphi_t+q-\lambda\varphi_{xx},\varphi_x+1)\right)_{xx}  =0.
\end{aligned}
\end{equation}
 Because $\varphi_x+1>0$, we have
\begin{equation}\label{q12}
\frac{\partial}{\partial \varphi_{xx}}L_0(\varphi_t+q-\lambda\varphi_{xx},\varphi_x+1)=
-\lambda L^\prime\Big(\frac{ \varphi_t+q-\lambda \varphi_{xx}}{\varphi_x+1}\Big),
\end{equation}
\begin{equation}\label{q14}
\frac{\partial}{\partial \varphi_{t}}L_0(\varphi_t+q-\lambda\varphi_{xx},\varphi_x+1)=
L^\prime\Big(\frac{ \varphi_t+q-\lambda \varphi_{xx}}{\varphi_x+1}\Big),
\end{equation}
and
\begin{equation}\label{q13}
\frac{\partial}{\partial \varphi_x}L_0(\varphi_t+q-\lambda\varphi_{xx},\varphi_x+1)=
\frac{\partial}{\partial \varphi_x}\left( L\Big(\frac{ \varphi_t+q-\lambda \varphi_{xx}}{\varphi_x+1}\Big)(\varphi_x+1)\right).
\end{equation}
Notice that
\begin{equation}\label{q23}
\begin{split}
&\left( \frac{\partial}{\partial \varphi_x}\left( L\Big(\frac{ \varphi_t+q-\lambda \varphi_{xx}}{\varphi_x+1}\Big)(\varphi_x+1)\right)\right)_x\\= &\left( L\Big(\frac{ \varphi_t+q-\lambda \varphi_{xx}}{\varphi_x+1}\Big)\right) _x- \left(  \frac{ \varphi_t+q-\lambda \varphi_{xx}}{\varphi_x+1}L^\prime\Big(\frac{ \varphi_t+q-\lambda \varphi_{xx}}{\varphi_x+1}\Big) \right)_x\\=&-\frac{ \varphi_t+q-\lambda \varphi_{xx}}{\varphi_x+1}\left(L^\prime\Big(\frac{ \varphi_t+q-\lambda \varphi_{xx}}{\varphi_x+1}\Big) \right)_x.
\end{split}
\end{equation}
Thus,   combining \eqref{q11}, \eqref{q12}, \eqref{q14}, \eqref{q13}, and \eqref{q23},  we get \eqref{phiEq3} from \eqref{phiEq}.
        \end{proof}
\begin{proof}[Proof of Proposition~\ref{prorel}] We first observe that \((u,m)\) satisfies the continuity equation, the second equation in \eqref{planningSGS},  from the definition of the functions $u$ and $m$ in \eqref{proreleq}. To complete the proof, it remains to address the Hamilton--Jacobi equation, the first equation in \eqref{planningSGS}. Since $(\varphi, q)$ solves Problem~\ref{problem} and $\varphi_x+1>0$, then by Proposition~\ref{El1}, we deduce that $\varphi$ and $q$ satisfy \eqref{phiEq}. According to the proof of Proposition~\ref{El1}, we know that \eqref{phiEq} is equivalent to \eqref{phiEq0}. Finally, by using \eqref{proreleq}, we get \eqref{HJux} from \eqref{phiEq0}, which implies that there exists a function, \(c:[0,T]\to\Rr\), only depending
on \(t\), such that
\begin{equation*}
\begin{aligned}
-u_t-\lambda u_{xx}+H(u_x)+V(x)-g(m) = c.
\end{aligned}
\end{equation*}
Hence, setting \(\vartheta(t)= \int_0^t c(s)\,\d s\) and \(\tilde u(x,t) = u(x,t) + \vartheta(t)\), we conclude that 
 $(\tilde u, m)$ solves both equations   in \eqref{planningSGS}; that  is, \((\tilde u,m)\) is a solution to Problem~\ref{planningSG}.
\end{proof}

       \subsection{Variational approach for the first-order MFP problem in dimension greater than 1.}\label{hrem}
                In this part, we describe the variational approach for the first-order planning problem for   $d>1$. We consider the MFP
                   \begin{equation}
                \label{MFGs-d>1}
                \begin{cases} 
                -u_t+H(Du)+V(x)=g(m)\quad &\text{in}\,\ (0,T)\times \Tt^d  \\ m_t-\div (mH_p(x,Du))=0 \quad &\text{in}\,\ (0,T)\times \Tt^d \\
                m(0,x)=m_0(x), \quad m(T,x)=m_T(x) &\text{in}\,\ \Tt^d.
                \end{cases} 
                \end{equation}
                As we will show next, this problem can be converted into a variational problem similar to the ones considered in \cite{DACOROGNA_GANGBO_diff} and \cite{DACOROGNA20151099}. As
in \cite{DACOROGNA_GANGBO_diff}, we denote by $\Lambda^k$ the set of differential $k$-forms over $\Rr^d$.
Using Poincar{\'e's}
lemma (see \cite[Theorem~1.22]{Csato2011ThePE}),  we deduce from the second equation in \eqref{MFGs-d>1}   that there exist a differential $(d-1)$-form, $\omega$,  and a vector function, $\textbf{a}:[0,T]\to\Rr^d$, such that
        \begin{equation}\label{higher}
        \begin{cases}
        m=\div \omega+1\\
        mD_{p}H(Du)=\omega_t+\textbf{a}(t).
        \end{cases}
        \end{equation}  
        From the second equation of the preceding system, we have
          \begin{equation}\label{int_ux-d>1}
        0=\int_{\Tt^d} Du\,\dx =\int_{\Tt^d}  D_vL\Big(\frac{ \omega_t+ \textbf{a}(t)}{\div \omega+1}\Big) \,\dx,
        \end{equation}
         where $L$ is the Legendre transform of $H$.
         
        Similarly to the one-dimensional case, it follows from
\eqref{higher} that the first equation in \eqref{MFGs-d>1} is equivalent to the following system of PDEs in terms of $\omega$ and $\textbf{a}$:
    \begin{equation}\label{phiEq0-d>1}
    \begin{split}
    &-\left( D_v\left( L\Big(\frac{ \omega_t+ \textbf{a}(t)}{\div \omega+1} \Big)\right)\right)_t + D\left( H\left(  D_v L\Big(\frac{ \omega_t+ \textbf{a}(t)}{\div \omega+1}\Big) \right)  \right)\\
&\quad+D(V)-D(\div \omega+1)Dg(\div \omega+1)=0.
    \end{split}
    \end{equation}
      Furthermore, arguing as in the proof of Proposition \ref{El1}, we deduce that \eqref{int_ux-d>1}--\eqref{phiEq0-d>1} is the system of Euler--Lagrange equations of  the functional 
       \begin{equation*}\label{varProL0-d>1}
       \int_{0}^{T}\!\!\!\int_{\mathbb{T}^d} \big(L_0(\omega_t+ \textbf{a}(t),\div \omega+1)-V\div \omega+G( \div \omega+1)\big)\,\dx\dt,
       \end{equation*}
       where $G$ and $L_0$  are corresponding extensions of \eqref{G-def} and \eqref{defLTilde} in the higher-dimensions.
        
        To complete the variational formulation of \eqref{MFGs-d>1}, it remains to impose the boundary conditions. In particular, from the first equation in \eqref{higher} and  the boundary conditions  in \eqref{MFGs-d>1}, it follows that we need to impose  conditions on $\div \omega$. Therefore,  we use Definition~2.2 in  \cite{DACOROGNA_GANGBO_diff}; that is,
        \begin{equation}\label{weak-bounday-d>1}
\int_{0}^{T}\!\!\!\int_{\Tt^d}(h_t,\div(\omega))+(\delta h,\omega_t)\,\dx\dt=\int_{\Tt^d}(h(T,x),m_T-1)-( h(0,x),m_0-1)\,\dx,
        \end{equation}
        where $\delta$ is the codifferential of the $(d-1)$-form $\omega_t$. Thus, we derive  the following variational problem
            \begin{equation}\label{mind>1}
        \inf\limits_{(\omega, \textbf{a})\in\mathcal{H}^p} \int_{0}^{T}\!\!\!\int_{\mathbb{T}^d} \big(L_0(\omega_t+ \textbf{a}(t),\div \omega+1)-V\div \omega+G( \div \omega+1)\big)\,\dx\dt,
         \end{equation}
         where $\mathcal{H}^p=\{\omega \in L^p([0,T]\times\Tt^d;\Lambda^{d-1}): \div(\omega)+1\geq 0, \,\, \omega \,\ \text{satisfies}\,\ \eqref{weak-bounday-d>1} \}\times L^p([0,T];\Rr^d)$.
        The problem \eqref{mind>1} is examined in detail  in  \cite{DACOROGNA20151099},\cite{DACOROGNA_GANGBO_diff}. The existence of solutions follows from  Proposition~3.1 in \cite{DACOROGNA_GANGBO_diff},  while uniqueness can be proven as in Proposition~\ref{uniq1} (see Section~\ref{sec2}).

        \section{Assumptions}
        \label{assumptions}
        This section presents the assumptions on the problem data that we need to establish our results.
        The first two assumptions concern the coupling function, $g$. We use the
first of these assumptions  to prove the uniqueness of a minimizer.
The second one ensures the convexity and coercivity of our variational problem, enabling the proof of the existence of a minimizer.        \begin{hyp}\label{assumtionOnG2} The function $G$  is  strictly convex.
        \end{hyp}

        \begin{hyp}\label{assumtionOnG1} The function $G$ is convex  and there exist constants, $\gamma>1$ and  $C>0$, such that \(G\)   satisfies the  growth condition 
                \begin{equation*}
                G(z)\geq C |z|^\gamma-C\enspace \text{ for all }z\in\Rr.
                \end{equation*}
        \end{hyp}
        
        Next, we impose a lower bound on the initial and terminal densities, $m_0$ and $m_T$.  
        \begin{hyp}\label{assumtionOnBoundsA} There exists a positive constant, $k_0$, such that
                \begin{equation*} m_0(x),\, m_T(x)\geq k_0>0 \enspace \text{ for all } x\in \Tt.
                \end{equation*} 
        \end{hyp}  
        
        Finally, we state an  assumption on the Hamiltonian, $H$, through a standard growth condition on the  Lagrangian, $L$. 
        \begin{hyp}\label{assumtionOnH2} There exist constants, $\beta >1$ and  $C>0$, such that   $L$ satisfies the  growth condition 
                \begin{equation*}
                L(w)\geq C|w|^\beta\enspace \text{ for all }w\in\Rr.
                \end{equation*}
        \end{hyp}
        
%       The next one we need to apply the transformation in Section~\ref{sec4}, which enables us to get  Euler--Lagrange equation from forward forward mean-field problem.
%       \begin{hyp}\label{assumtionOnH3} The Hamiltonian $H(p)$ is an even function.
%       \end{hyp}
%       
        \begin{remark} Given \(\alpha>1\), the Hamiltonian
                \begin{equation*}
                H(p)=(1+p^2)^\frac{\alpha}{2},\quad p\in\Rr
                \end{equation*}
                satisfies both Assumption \ref{assumtionOnH2} and the conditions in the statement of Problem~\ref{planningSG}.
        \end{remark}

        \section{Notation and Preliminaries}\label{preli}
        %\section{A Variational Approach}
        We begin this section  by introducing a suitable notation that leads to a more concise formulation of our problems. Then, we explore the properties of the function $L_0$ defined in \eqref{defLTilde} that  are crucial for studying our variational problems. Finally, we establish the convexity and lower semi-continuity of some  functions that we use in the sequel. 

Set 
        \begin{equation}
        \label{defL}
        \begin{split}
        L_1(q,z,y)=L_0(z+q,y),\\
        L_2(q,z,y, \theta)=L_0(z+q-\theta,y).
        \end{split}
        \end{equation}
                Using this notation, the functional in \eqref{prof} can be written as 
                \begin{equation}\label{varProbNoj}
                \mathcal{I}_1[\varphi,q]= \int_{0}^{T}\!\!\!\int_{\mathbb{T}}\big(L_1(q,\varphi_t,\varphi_x+1)-V\varphi_{x}+G( \varphi_x+1)\big)\,\dx\dt
                \end{equation}
         for   $\lambda=0$, and as                 \begin{equation}\label{varProbNoj2}
                \mathcal{I}_2[\varphi,q]= \int_{0}^{T}\!\!\!\int_{\mathbb{T}}\big(L_2(q,\varphi_t, \varphi_x+1,\varphi_{xx})-V\varphi_{x}+G(
\varphi_x+1)\big)\,\dx\dt
                \end{equation}
for $\lambda=1$.

Next, we prove some properties of $L_0$, $L_1$,  and $L_2$. Then, relying on these properties, we show that the functionals in \eqref{varProbNoj} and \eqref{varProbNoj2} are sequentially weakly lower semi-continuous (see Lemma~\ref{profunc} below).
        
        \begin{lem}\label{propConvexLtil}The function $L_0$ in \eqref{defLTilde} is convex in $\Rr\times\Rr_0^+$.
        \end{lem}
        \begin{proof} As we mentioned before (see \eqref{L}), the function
  $L$   is well defined and convex. Furthermore, because $L_0(z,y)=L\big( \tfrac{z}{y}\big)y$ for all \(z\in\Rr\) and  $y>0$, $L_0$ is convex in $\Rr\times\Rr^+$ (see \cite[Lemma 2]{DaMa08}). 

It remains to prove that given \(\lambda \in (0,1)\), \(z_1\in\Rr\), and \(z_2\in\Rr\), the inequality 
                \begin{equation}\label{convexityL}
                L_0(\lambda (z_1,y_1)+(1-\lambda)(z_2,y_2))\leq \lambda L_0 (z_1,y_1)+(1-\lambda)L_0(z_2,y_2)
                \end{equation}
                holds if either  $y_1=0$ or $y_2=0$. 
Assume that $y_1=0$. Then, if     $z_1\neq0$ or  $y_2=0$ with $z_2\neq0$, the right-hand side of \eqref{convexityL} equals $+\infty$; thus,   the inequality \eqref{convexityL} holds in both cases. If $z_1=0$ and $(z_2, y_2)=(0,0)$, then \eqref{convexityL} becomes $0\leq0$, while  if $z_1=0$, $y_2\neq0$, and $z_2\in\Rr$, then  \eqref{convexityL} reduces to $L\big( \frac{z_2}{y_2}\big)(1-\lambda)y_2\leq(1-\lambda)L\big( \frac{z_2}{y_2}\big)y_2$;
                thus,   
inequality \eqref{convexityL} also holds in these two cases. Finally, we note that the \(y_2=0\) case can be treated as the $y_1=0$ case.   
        \end{proof}
        \begin{cor}\label{propConvexL}  The functions $L_1$ and $L_2$ defined in \eqref{defL} are convex in $\Rr^{2}\times\Rr_0^+$ and $\Rr^{2}\times\Rr_0^+\times\Rr$, respectively.
        \end{cor}
\begin{proof}
 We first  observe that if    $f:\Rr^n\to \Rr$ is a convex function, then the function \(h\)  defined for \((\textbf{x},s)\in \Rr^{n+1}\) by $h(\textbf{x},s)=f(\textbf{x}+s)$ is convex in $\Rr^{n+1}$. Thus, the claim follows from this observation combined with  Lemma~\ref{propConvexLtil}.
\end{proof}
        \begin{lem}\label{proLowL} Suppose that Assumption \ref{assumtionOnH2} holds for some \(\beta>1\). Then,  the functions $L_1$ and $L_2$ defined in \eqref{defL} are lower semi-continuous in  $\Rr^{2}\times\Rr_0^+$ and $\Rr^{2}\times\Rr_0^+\times\Rr$, respectively.
        \end{lem}

        \begin{proof} Here, we  only prove the lower semi-continuity of  $L_1$. The lower semi-continuity of  $L_2$ can be handled similarly. Because $L_1$ is convex, $L_1$  is lower semi-continuous in the interior of its effective domain. Thus, it remains to prove that for any $(q,z)\in\Rr^{2}$  and for any sequence  $(q_n,z_n,y_n)\in\Rr^{2}\times\Rr_0^+$ with  $(q_n,z_n,y_n)\to(q,z,0)$, we have
                \begin{equation}\label{lowSemLEq}
                L_1(q,z,0)\leq\liminf\limits_{n\to\infty} L_1(q_n,z_n,y_n).
                \end{equation}

Suppose that  $z+q\neq0$,
                and let $S=\{n\in\Nn\!:\,y_n\neq0\}$ be the set of natural numbers defining 
the non-zero terms of the sequence $\{y_n\}$.                
                 We first note that \(z_n + q_n \not=0\) for all $n\in\Nn$ sufficiently
large. Moreover, if $S$ has finite cardinality, then  \eqref{lowSemLEq}
holds because  
                $L_1(q_n,z_n,y_n)=+\infty$ for all $n\in\Nn$ sufficiently large.  Otherwise, if $S$ has infinite cardinality, then, using Assumption~\ref{assumtionOnH2}, we have                           \begin{equation*}
                \begin{split}
                \liminf\limits_{n\to\infty} L_1(q_n,z_n,y_n)=\liminf\limits_{\substack{n\to\infty\\n\in
S}} L\big( \tfrac{z_n+q_n}{y_n}\big) y_n\geq&\lim\limits_{\substack{n\to\infty\\n\in S}}C\frac{|z_n+q_n|^\beta}{y_n^{\beta-1}} =+\infty.
                \end{split}
                \end{equation*}
                Therefore, we deduce that \eqref{lowSemLEq} holds in the $z+q\neq0$ case. Finally, if $z+q=0$, it is enough to prove that $ L_1(q_n,z_n,y_n)\geq 0$ for all $n\in\Nn$ sufficiently
large, which  follows from the definition of \(L_1\) and Assumption~\ref{assumtionOnH2}.
        \end{proof}
        
        \begin{lem}\label{profunc} Suppose that  Assumptions \ref{assumtionOnG1} and \ref{assumtionOnH2} hold  for some  \(\gamma>1\) and \(\beta>1\), respectively. Then, the functionals 
                \begin{equation}\label{profunc1}
                (v_1,v_2,v_3)\mapsto \int_{0}^{T}\!\!\!\int_{\mathbb{T}}\big(L_1(v_1,v_2, v_3)-V(v_3-1)+G( v_3)\big)\,\dx\dt
                \end{equation}  
                and 
                \begin{equation}\label{profunc2}
                (v_1,v_2,v_3,v_4)\mapsto \int_{0}^{T}\!\!\!\int_{\mathbb{T}}\big(L_2(v_1,v_2,v_3, v_4)-V(v_3-1)+G( v_3)\big)\,\dx\dt
                \end{equation}          
                are sequentially lower semi-continuous with respect to the weak convergence in 
                $L^1([0,T]\times \Tt)^2\times L^1([0,T]\times \Tt;\Rr_0^+)$ and in $L^1([0,T]\times \Tt)^2\times  L^1([0,T]\times \Tt;\Rr_0^+) \times L^1([0,T]\times \Tt)$, respectively.
        \end{lem}

        \begin{proof} We only prove the sequential lower semi-continuity of the functional in \eqref{profunc1}. The functional in \eqref{profunc2} can be treated similarly. 

First, we notice that the  functional
                \begin{equation*}
                v_3\mapsto\int_{0}^{T}\!\!\!\int_{\mathbb{T}}V(v_3  -1 )\,\dx\dt
                \end{equation*} 
                is continuous with respect to the weak convergence in $L^1([0,T]\times \Tt)$ because  $V\in C^1(\Tt)$. Thus, to conclude, it suffices  to prove that the functional
                \begin{equation}\label{funcNoV}
                (v_1,v_2,v_3)\mapsto\int_{0}^{T}\!\!\!\int_{\mathbb{T}}\big( L_1(v_1,v_2,v_3)+G( v_3)\big)\,\dx\dt
                \end{equation}
                is sequentially lower semi-continuous with respect to the weak convergence in 
                $L^1([0,T]\times \Tt)^2\times L^1([0,T]\times \Tt;\Rr_0^+)$. Because  \(G(\cdot)\geq -C\) by Assumption
\ref{assumtionOnG1} and   $L_1(\cdot,\cdot,\cdot)\geq 0$, we have
                \begin{equation}\label{funcBoBelow}
                L_1(z_1,z_2,z_3)+G( z_3)\geq-C\enspace \text{ for all } z=(z_1,z_2,z_3)\in \Rr^2 \times \Rr^+_0.
                \end{equation}
                Moreover, from Corollary \ref{propConvexL} and Lemma~\ref{proLowL}, we deduce that the function
                $(z_1,z_2,z_3) \mapsto L_1(z_1,z_2,\allowbreak z_3-1)+G( z_3)$ is lower semi-continuous and convex in $\Rr^2 \times \Rr^+_0$. These two properties and the lower bound in \eqref{funcBoBelow} enable us to use Theorem~5.14 in \cite{FoLe07}  to conclude   the lower semi-continuity of the functional in \eqref{funcNoV} stated above, which concludes the proof of the lemma.
        \end{proof}

        \section{The Variational Approach  For The First-Order Planning Problem}
        \label{sec2}
        Here, we focus on the variational problem for the first-order case; that is, $\lambda=0$ in Problem~\ref{problem}. The approach considered here is different from the one that we described in Subsection ~\ref{hrem}.   In higher dimensions (see Subsection ~\ref{hrem}), we set the boundary conditions in a weak sense (see \eqref{weak-bounday-d>1}) because of lack of regularity. Here, however, we are in the one-dimensional case, where we can impose boundary conditions exactly.
        
        We begin by supplementing the discussion in Section~\ref{sec1} by examining the boundary conditions in \eqref{boundaryP}. We then introduce the precise functional setting for this variational problem. Finally, using the direct method in the calculus of variations, we prove the existence of a solution. We also establish the uniqueness of the solution. 
        
        Using \eqref{boundaryP},  the first equation in \eqref{um_pi}, and
the average condition
\begin{equation}\label{phi_int_0}
\int_{\Tt}\varphi(t,x)\,\dx=0 \enspace \text{ for all \(t\in[0,T]\)}
        \end{equation}
 derived in Section~\ref{sec1}, we conclude that
\begin{equation}\label{eq:ffi0T}
\begin{aligned}
\varphi(0,x)=\int_{0}^{x}(m_0(\tau)-1)\,\d \tau - \iota_0 \enspace \text{ and } \enspace \varphi(T,x)=\int_{0}^{x}(m_T(\tau)-1)\,\d
\tau - \iota_T, 
\end{aligned}
\end{equation}
where 
\begin{equation}\label{eq:iotas}
\begin{aligned}
\iota_0 =\int_{\Tt}\int_{0}^{x}(m_0(\tau)-1)\,\d \tau\dx \enspace \text{
and } \enspace  \iota_T = \int_{\Tt}\int_{0}^{x}(m_T(\tau)-1)\,\d \tau\dx.
\end{aligned}
\end{equation}

        Thus, for $p>1$, we introduce the set
                \begin{equation}\label{def-A_0}
         \begin{split}
         \mathcal{A}_{0}^p= &\left\{\varphi\in W^{1,p}([0,T]\times\Tt) :\,\varphi_x+1\geq 0;\,\, \int_{\Tt}\varphi(t,x)\,\dx=0;\right. \\  &\quad\quad\quad\quad\left.  \varphi(0,x)=\int_{0}^{x}\big(m_0(\tau)-1\big)\,\d \tau- \iota_0;\,\,\varphi(T,x)=\int_{0}^{x}\big(m_T(\tau)-1\big)\,\d \tau- \iota_T\right\}
         \end{split}
         \end{equation}
         and note that the admissible set $\mathcal{B}_{0}^p$ in \eqref{B00} can be written as
                \begin{equation*}
         \begin{split}
         \mathcal{B}_{0}^p=\mathcal{A}_{0}^p\times L^p([0,T]).
         \end{split}
         \end{equation*}
        
Recalling the notation in \eqref{defL}, this analysis leads us to the variational problem        \begin{equation}\label{varProbphii}
        \min\limits_{(\varphi,q)\in\mathcal{B}_{0}^p}\mathcal{I}_1[\varphi,q]= \min\limits_{(\varphi,q)\in \mathcal{B}_{0}^p} \int_{0}^{T}\!\!\!\int_{\mathbb{T}}\big(L_1(q,\varphi_t,\varphi_x+1)-V\varphi_x+ G( \varphi_x+1)\big)\,\dx\dt,
        \end{equation}
which corresponds  to Problem~\ref{problem} with  $\lambda=0$.         
        The next proposition proves the existence of a solution to this minimization  problem.        \begin{pro}\label{ex1}Suppose that Assumptions~\ref{assumtionOnG1}--\ref{assumtionOnH2} hold  for some \(\beta>1\) and \(\gamma>1\) (see Section~\ref{assumptions}). Then, the variational problem in \eqref{varProbphii} with \(p=\frac{\beta\gamma}{\beta+\gamma-1}  \) has a solution $(\varphi,q)\in \mathcal{B}_{0}^ \frac{\beta\gamma}{\beta+\gamma-1}$.
        \end{pro}

\begin{proof}
Let
                $\sigma=\frac{\beta\gamma}{\beta+\gamma-1}$, and note that \(\sigma>1\). First, we prove
that there exist positive constants, $c$ and $C$, such that 
                \begin{equation}\label{funcBound}
                -c\leq\inf\limits_{\substack{(\varphi,q)\in \mathcal{B}_0^\sigma}}\mathcal{I}_1[\varphi,q]\leq
C.
                \end{equation}

To prove the lower bound in   \eqref{funcBound}, 
we take  any \((\varphi,q)\in \mathcal{B}_0^\sigma\) and   observe that     
                because $\varphi_x+1\geq0$ and $V\in C^1(\Tt)$, we have
                \begin{equation}\label{ineqForLeft}
                \begin{split}
                \left| \int_{0}^{T}\!\!\!\int_{\mathbb{T}}V\varphi_x\,\dx\dt\right| &\leq\max\limits_{\Tt}|V| \int_{0}^{T}\!\!\! \int_{\mathbb{T}}|\varphi_x|\,\dx\dt \\&\leq\max\limits_{\Tt}|V| \int_{0}^{T}\!\!\!\int_{\mathbb{T}} \big((\varphi_x+1)+1\big)\,\dx\dt= 2T\max\limits_{\Tt}|V|.
                \end{split}
                \end{equation}
                Moreover,    the convexity of \(G\) and Jensen's
inequality yield  \(\int_{0}^{T}\!\!\!\int_{\mathbb{T}}
G( \varphi_x+1)\,\dx\dt \geq TG(1)\).     Then,  because $L_1\geq 0$ by Assumption~\ref{assumtionOnH2}
and the definition of \(L_1\) (see \eqref{defL} and \eqref{defLTilde}), we conclude that the lower bound in  \eqref{funcBound} holds with $c=T|2\max\limits_{\Tt}|V|-G(1)|$. 
                
                To prove the upper bound inequality in \eqref{funcBound}, we take  $q^0(t)=0$ and 
                \begin{equation}\label{phiTest}
                \varphi^0(t,x)=\frac{T-t}{T}\bigg(\int_{0}^{x}\big(m_0(\tau)-1\big)\,\d  \tau-\iota_0\bigg)+\frac{t}{T}\bigg(\int_{0}^{x}\big(m_T(\tau)-1\big)\,\d \tau -\iota_T\bigg),
                \end{equation}
                where \(i_0\) and \(i_T\) are the constants in \eqref{eq:iotas}. Note that $(\varphi^0,q^0)\in \mathcal{B}_0^\sigma$ with
\begin{equation*}
\begin{aligned}
&\varphi_t^0 (t,x)= -\frac{1}{T}\bigg(\int_{0}^{x}
\big(m_0(\tau)-m_T(\tau)\big)\,\d
 \tau-\iota_0 + \iota_T\bigg),  \\
 & \varphi_x^0(t,x) + 1= \frac{T-t}{T} m_0(x) + \frac{t}{T} m_T(x)
\end{aligned}
\end{equation*}
for all \((t,x) \in [0,T]\times \Tt\).                

By Assumption \ref{assumtionOnBoundsA}  and the smoothness of the initial-terminal densities, $m_0$ and $m_T$, there exist  positive constants, \(k_0\) and  $k_1$, such that
                \begin{equation}\label{boabove}
                k_0 \leq m_0(x),m_T(x)\leq k_1 \enspace \hbox{ for all } x\in\Tt.
                \end{equation} 
                Hence,
we have the following uniform bounds on  \( [0,T]\times \Tt\):
\begin{equation}\label{eq:bddffixt}
\begin{aligned}
\max_{[0 ,T]\times \Tt} |\varphi_t^0| \leq \frac{1}{T}\big( 2k_1 +\ 2T(k_1+1)\big) \enspace \text{ and } \enspace k_0 \leq \min_{[0 ,T]\times \Tt} (\varphi_x^0 
+1) \leq \max_{[0 ,T]\times \Tt} (\varphi_x^0 +1) \leq k_1.\end{aligned}
\end{equation}
Moreover, for \(k_2 = \tfrac{1}{T}( 2k_1 +\ 2T(k_1+1))\),
from  \eqref{boabove}--\eqref{eq:bddffixt},  the convexity of \(G\), and the definition of \(L_1\),  we deduce that
\begin{equation*}
\begin{aligned}
&\max_{[0 ,T]\times \Tt} G(\varphi^0_x+1)\leq \max_{z\in[k_0,k_1]} |G(z)|,\\
&\max_{[0 ,T]\times \Tt} L_1(0,\varphi^0_t,\varphi^0_x+1) = \max_{[0 ,T]\times \Tt} L\Big(\frac{\varphi^0_t}{\varphi^0_x+1}\Big)(\varphi^0_x+1) \leq k_1 \max_{w\in\big[0,\tfrac{k_2}{k_0}\big]} |L(w)|.
\end{aligned}
\end{equation*}
Thus, recalling \eqref{ineqForLeft},  the upper bound in \eqref{funcBound} holds with
\begin{equation*}
\begin{aligned}
C= k_1T
\max_{\big[0,\tfrac{k_2}{k_0}\big]} |L| +2T\max\limits_{\Tt}|V| +T\max_{[k_0,k_1]} |G|.  
\end{aligned}
\end{equation*}

Next, we prove the existence of a minimizer as stated in the proposition.                  
                Let $(\varphi^n,q_n)_{n\in\Nn}\subset\mathcal{B}_0^\sigma$ be a minimizing sequence for the variational problem \eqref{varProbphii} with \(p=\frac{\beta\gamma}{\beta+\gamma-1}
 \); that is,
                \begin{equation*}
                \lim\limits_{n\to\infty}\mathcal{I}_1[\varphi_n,q_n]=\inf\limits_{\substack{(\varphi,q)\in \mathcal{B}_0^\sigma}}\mathcal{I}_1[\varphi,q]=    \ell_1.
                \end{equation*}
From  \eqref{funcBound}, the non-negativity of \(L_1\), and Assumption~\ref{assumtionOnG1}, we conclude that there exists a positive constant, \(C_0\), independent of \(n\), such that
\begin{equation*}
\begin{aligned}
\int_{0}^{T}\!\!\!\int_{\Tt} \big(-V\varphi^n_x +C|\varphi^n_x+1|^{\gamma}-C\big)\,\dx\dt\leq C_0.
\end{aligned}
\end{equation*}
Hence,  H\"older's inequality  and the smoothness of \(V\) yield         %
\begin{equation}
\label{seqboundphix}
\begin{aligned}
\sup_{n\in\Nn}\int_{0}^{T}\!\!\!\int_{\Tt} |\varphi^n_x+1|^{\gamma}\,\dx\dt<+\infty.
\end{aligned}
\end{equation}

Using once more   \eqref{funcBound},  \eqref{seqboundphix}, the non-negativity  of \(z\mapsto CG(z)+C\) by Assumption~\ref{assumtionOnG1} and of \(L_1\),
 we may assume that \(C_0\) is such that
              \begin{equation}\label{seqboundLLL}
                \int_{0}^{T}\!\!\!\int_{\Tt}L_1(q_n,\varphi^n_t,\varphi^n_x)\,\dx\dt \leq C_0
                \end{equation}
for all \(n\in\Nn\). In particular, by the definition of \(L_1\), we have  $\varphi^n_x+1>0$
a.e.~in $U_n=\{(t,x)\in[0,T]\times\Tt:\varphi^n_t+q_n\neq0\}$. Moreover,
using Assumption~\ref{assumtionOnH2}, %
\begin{equation}\label{seqboundL}
\begin{aligned}
C_0 \geq \int_{0}^{T}\!\!\!\int_{\Tt}L_1(q_n,\varphi^n_t,\varphi^n_x)\,\dx\dt = \int\!\!\!\int_{U_n}L_1(q_n,\varphi^n_t,\varphi^n_x)\,\dx\dt \geq C 
\int\!\!\!\int_{U_n}\frac{|\varphi^n_t+q_n|^\beta}{|\varphi^n_x+1|^{\beta-1}}\,\dx\dt.
\end{aligned}
\end{equation}
On the other hand, by Young's inequality, taking into account that $\tfrac{\beta+\gamma-1}{\gamma}>1$,
 we obtain 
                \begin{equation*}
                \begin{split}
                &\int_{0}^{T}\!\!\!\int_{\Tt}|\varphi^n_t+q_n|^\sigma\,\dx\dt= \int\!\!\!\int_{U_n}|\varphi^n_t+q_n|^\sigma \,\dx\dt =\int\!\!\!\int_{U_n}\frac{|\varphi^n_t+q_n|^\sigma}{|\varphi^n_x+1|^{\frac{(\beta-1)\gamma}{\beta+\gamma-1}}} |\varphi^n_x+1|^{\frac{(\beta-1)\gamma}{\beta+\gamma-1}}\,\dx\dt\\& \quad\leq\frac{\gamma}{\beta+\gamma-1} \int\!\!\!\int_{U_n} \frac{|\varphi^n_t+q_n|^{\sigma\frac{\beta+\gamma-1}{\gamma}}}{|\varphi^n_x+1|^{\beta-1}}\,\dx\dt+\frac{\beta-1}{\beta+\gamma-1} \int_{0}^{T}\!\!\!\int_{\Tt}|\varphi^n_x+1|^\gamma\,\dx\dt.
                \end{split}
                \end{equation*}
From the preceding estimate, the identity  \(\sigma\tfrac{\beta+\gamma-1}{\gamma} = \beta\), and the uniform bounds in   \eqref{seqboundphix} and  \eqref{seqboundL}, we infer that 
\begin{equation}\label{seqboundphit}
\begin{aligned}
\sup_{n\in\Nn}\int_{0}^{T}\!\!\!\int_{\Tt} |\varphi^n_t+q_n|^\sigma\,\,\dx\dt<+\infty.
\end{aligned}
\end{equation}

Next, we observe that because $\varphi^n\in\mathcal{A}_0^\sigma$, \eqref{phi_int_0} holds with \(\varphi\) replaced by \(\varphi^n\);  thus, we have
                        \begin{equation}\label{derivph}
                \begin{split}
        \int_{\Tt}\varphi^n_t\,\dx= \frac{\partial}{\partial t} \int_{\Tt}\varphi^n\,\dx=0.
                \end{split}
                \end{equation}
                Since each $q_n$ depends only on $t$, by using Jensen's inequality, we get from  \eqref{derivph} and \eqref{seqboundphit} that 
                \begin{equation}\label{qLnorm}
                \begin{split}
                \sup\limits_{n\in\Nn} \int_0^T |q_n|^\sigma\,\dt&=\sup\limits_{n\in\Nn}\int_{0}^{T}\left| \int_{\Tt}q_n\,\dx\right|^\sigma\dt =\sup\limits_{n\in\Nn}\int_{0}^{T}\left| \int_{\Tt}\big(\varphi^n_t+q_n\big)\,\dx\right| ^\sigma\dt \\&\leq\sup\limits_{n\in\Nn} \int_{0}^{T}\!\!\!\int_{\Tt}|\varphi^n_t+q_n|^\sigma \,\dx\dt<+\infty.
                \end{split}
                \end{equation}
Hence, using  \eqref{seqboundphit}
and  \eqref{qLnorm}, we conclude that 
                \begin{equation}\label{phitrnorm}
               \sup\limits_{n\in\Nn}
\int_{0}^{T}\!\!\!\int_{\Tt}|\varphi^n_t|^\sigma \,\dx\dt<+\infty.
                \end{equation}

Finally, by Poincar{\'e}--Wirtinger's inequality and \eqref{phi_int_0}
 with \(\varphi\) replaced by \(\varphi^n\), and observing that \(\gamma>\sigma>1\), we deduce from \eqref{seqboundphix}
and \eqref{phitrnorm}
that \(\{\varphi_n\}_{n\in\Nn}\) is a bounded sequence in  $W^{1,\sigma}([0,T]\times\Tt)$. Thus, up to extracting of a subsequence if necessary, there
exist  $\bar{\varphi}\in W^{1,\sigma}([0,T]\times\Tt)$ and $\bar{q}\in L^\sigma([0,T])$ such that                        \begin{equation}\label{weaklyphi_L}
                \varphi^n\rightharpoonup\bar{\varphi}\quad \text{weakly
in} \,\ W^{1,\sigma}([0,T]\times\Tt)
                \end{equation}
 and                 
                        \begin{equation}\label{weaklyq}
                q^n\rightharpoonup\bar{q}\quad \text{weakly in} \,\ L^\sigma([0,T]).
                \end{equation}

Next, we show that  $\bar{\varphi}\in\mathcal{A}_0^\sigma$. 
                 By  Jensen's inequality and \eqref{phi_int_0}
 with \(\varphi\) replaced by \(\varphi^n\), we get
                \begin{equation*}
        \begin{split}
\int_{0}^{T}\left| \int_{\Tt}\bar\varphi\,\dx\right| ^\sigma\dt=\int_{0}^{T}\left|
\int_{\Tt}(\bar\varphi-\varphi^n)\,\dx\right|^\sigma\dt \leq \int_{0}^{T}\!\!\!\int_{\Tt}|\bar\varphi-\varphi^n|^\sigma\,\dx\dt.
        \end{split}
        \end{equation*}
                On the other hand, by the Rellich--Kondrachov theorem and \eqref{weaklyphi_L}, we have \(\varphi^n
\to \bar\varphi\)  in \(L^\sigma([0,T]\times \Tt)\), which together with the preceding estimate yields   $\int_{\Tt}\bar\varphi\,\dx=0$  a.e.~in $[0,T]$. Moreover, using the properties of weak convergence and the trace theorem in Sobolev spaces
(see, for example,  \cite[Theorem~1 in Section~5.5]{E6}), we conclude that \(\bar\varphi\)   also satisfies the remaining conditions defining the set \(\mathcal{A}_0^\sigma\). Thus, $\bar{\varphi}\in\mathcal{A}_0^\sigma$ and  $(\bar{\varphi},\bar{q})\in\mathcal{B}_0^\sigma$.

                Finally, by Lemma~\ref{profunc} with \(v_1=\bar q\), \(v_2=\bar \varphi_t\), and \(v_3=\bar\varphi_x +1\), we have from \eqref{weaklyphi_L}--\eqref{weaklyq} that
                \begin{equation*}
                \mathcal{I}_1[\bar{\varphi},\bar{q}]\leq\liminf\limits_{n\to\infty} \mathcal{I}_1[\varphi_n,q_n]=\ell_1\leq \mathcal{I}_1[\bar{\varphi},\bar{q}].
                \end{equation*} 
Hence,
 $(\bar{\varphi},\bar{q})\in\mathcal{B}_0^\sigma$ is a minimizer of \(\mathcal{I}_1\) on \(\mathcal{B}_0^\sigma\), which concludes the proof.                  
        \end{proof}
        
        Next, we prove the uniqueness of solutions of the variational problem in \eqref{varProbphii}.
        \begin{pro}\label{uniq1}Suppose that Assumptions \ref{assumtionOnG2}--\ref{assumtionOnH2}
hold  for some \(\beta>1\) and \(\gamma>1\) (see Section~\ref{assumptions}). Then, the variational problem in \eqref{varProbphii} has at most one solution.
        \end{pro}
        \begin{proof}As shown in Proposition~\ref{ex1}, we have 
\begin{equation*}
\begin{aligned}
\ell_1=\min\limits_{(\varphi,q)\in\mathcal{B}_{0}^p}\mathcal{I}_1[\varphi,q]\in\Rr.
\end{aligned}
\end{equation*}

Let $(\varphi^1,q_1)$ and $(\varphi^2,q_2)$ solve \eqref{varProbphii}; that is,
                \begin{equation*}
                \mathcal{I}_1[\varphi^1,q_1]=\mathcal{I}_1[\varphi^2,q_2]=\mathcal{\ell}_1.
                \end{equation*}
                Setting $\bar{\varphi}=\frac{1}{2}(\varphi^1+\varphi^2)$ and $\bar{q}=\frac{1}{2}(q_1+q_2)$, by the convexity of $L_1$ and $G$, we have
                \begin{equation}\label{eq:mincxty}
                \ell_1\leq \mathcal{I}_1[\bar{\varphi},\bar{q}]\leq\tfrac{1}{2}\mathcal{I}_1[\varphi^1,q_1]+\tfrac{1}{2}\mathcal{I}_1[\varphi^2,q_2]=\ell_1.
                \end{equation}
                
Arguing as in the proof of
Proposition~\ref{ex1} to deduce \eqref{seqboundLLL}, we conclude that 
\begin{equation*}
\begin{aligned}
L_1(q_1,\varphi^1_t,\varphi^1_x),\, L_1(q_2,\varphi^2_t,\varphi^2_x), \, L_1(\bar{q},\bar{\varphi}_t,\bar{\varphi}_x) <+\infty\enspace \hbox{a.e.~in } [0,T]\times\Tt.
\end{aligned}
\end{equation*}
In particular, recalling  \eqref{defL} and \eqref{defLTilde}, we have
\begin{equation}
\label{eq:defLL0}
\begin{aligned}
&\varphi^1_t+q_1=0 \enspace\text{ a.e.~in } \{(t,x)\in [0,T]\times\Tt:{\varphi}_x^1= -1 \},\\
& \varphi^2_t+q_2=0 \enspace\text{ a.e.~in } \{(t,x)\in [0,T]\times\Tt:{\varphi}_x^2=-1 \},\\
& \bar\varphi_t + \bar q=0 \enspace\text{ a.e.~in } \{(t,x)\in [0,T]\times\Tt:\bar{\varphi}_x= -1 \}. 
\end{aligned}
\end{equation}
Moreover,
\eqref{eq:mincxty} yields                \begin{equation}\label{sumalsomin}
                \begin{split}
                \int_{0}^{T}\!\!\!\int_{\Tt}\Big(\tfrac{1}{2}L_1(q_1,\varphi^1_t,\varphi^1_x)+&\tfrac{1}{2}L_1(q_2,\varphi^2_t,\varphi^2_x)-L_1(\bar{q},\bar{\varphi}_t,\bar{\varphi}_x)
                \\+&\tfrac{1}{2}G\left( \varphi^1_x+1\right)+\tfrac{1}{2}G\left( \varphi^2_x+1\right)-G( \bar{\varphi}_x+1)\Big)\,\dx\dt=0.
                \end{split}
                \end{equation}
                
Using the  convexity of $L_1$  and $G$ once more, we get from \eqref{sumalsomin} that                
\begin{equation}\label{eq:cxcomb=0}
                \begin{cases}
                \frac{1}{2}L_1(q_1,\varphi^1_t,\varphi^1_x)+\frac{1}{2}L_1(q_2,\varphi^2_t,\varphi^2_x)-L_1(\bar{q},\bar{\varphi}_t,\bar{\varphi}_x)=0 ,\quad &\text{a.e in } [0,T]\times\Tt,  \\ \frac{1}{2}G\left( \varphi^1_x+1\right)+\frac{1}{2}G\left( \varphi^2_x+1\right)-G\left( \bar{\varphi}_x+1\right)=0 ,\quad &\text{a.e in } [0,T]\times\Tt .
                \end{cases}
                \end{equation}

The second identity in \eqref{eq:cxcomb=0} and the strict convexity  of  $G$ implies that
\begin{equation}
\label{eq:varphix=}
\begin{aligned}
\varphi^1_x=\varphi^2_x=\bar{\varphi}_x \text{ a.e.~in } [0,T]\times\Tt.
\end{aligned}
\end{equation}
 Then, using the first identity in \eqref{eq:cxcomb=0},   \eqref{defL}, \eqref{defLTilde}, and \eqref{eq:defLL0}, we conclude that
                \begin{equation}\label{uniqenesssys}
                \begin{cases}
                \frac{1}{2}L\Big ( \frac{q_1+\varphi^1_t}{\bar{\varphi}_x+1}\Big) +\frac{1}{2}L\Big( \frac{q_2+\varphi^2_t}{\bar{\varphi}_x+1}\Big)-L\Big( \frac{\bar{q}+\bar{\varphi}_t}{\bar{\varphi}_x+1}\Big) =0,\,&\text{a.e in } \{(t,x)\in [0,T]\times\Tt:\bar{\varphi}_x\neq -1 \} \\
                \varphi^1_t+q_1=\varphi^2_t+q_2=0,\,&\text{a.e in } \{(t,x)\in [0,T]\times\Tt:\bar{\varphi}_x= -1 \}.
                \end{cases}
                \end{equation}
                Because $L$ is strictly convex, we obtain from \eqref{uniqenesssys} that
                \begin{equation}\label{general_eq1}
                \varphi^1_t+q_1=\varphi^2_t+q_2 \enspace\text{ a.e~in }  [0,T]\times\Tt .
                \end{equation}
                Integrating  \eqref{general_eq1} over $\Tt$, we get
                \begin{equation}\label{eqfor-q1}
                q_2-q_1 =\int_{\Tt}\Big(\varphi^1_t-\varphi^2_{t}\Big)\,\dx.
                \end{equation}
                Because $\int_{\Tt}\varphi^1_t\,\dx=\int_{\Tt}\varphi^2_t\,\dx=0$ a.e.~in \([0,T]\) (see \eqref{derivph}), we obtain from \eqref{eq:varphix=} and \eqref{eqfor-q1}  that $q_1=q_2$ and  $\varphi^1=\varphi^2$ a.e.~in $[0,T]\times\Tt$.
        \end{proof}
        
        \begin{teo}\label{exist1} Suppose that Assumptions~\ref{assumtionOnG2}--\ref{assumtionOnH2}
hold  for some \(\beta>1\) and \(\gamma>1\).
Then, the variational problem in \eqref{varProbphii} with \(p=\frac{\beta\gamma}{\beta+\gamma-1}
 \) has a unique solution $(\varphi,q)\in \mathcal{B}_0^ \frac{\beta\gamma}{\beta+\gamma-1}$.
        \end{teo}
        \begin{proof}The proof follows from Propositions~\ref{ex1} and \ref{uniq1}.
        \end{proof}
        
        \section{The Variational Approach  For The Second-Order Planning Problem}
        \label{sec3}
        Here, we study the variational problem corresponding to $\lambda=1$ in Problem~\ref{problem}. According to the discussion in Section~\ref{sec1} and arguing as in the previous section, the appropriate functional set for \(\varphi\) is the set
                \begin{equation*}
        \mathcal{A}_1^p=\left\{\varphi\in \mathcal{A}_0^p: \varphi_{xx}\in L^{p}([0,T]\times\Tt)\right\},
        \end{equation*}
in which case the admissible set  $\mathcal{B}_{1}^p$  in  \eqref{B001}  for $(\varphi,q)$ can be written as        \begin{equation*}
        \begin{split}
        \mathcal{B}_{1}^p=\mathcal{A}_{1}^p\times L^p([0,T]).
        \end{split}
        \end{equation*}

        Recalling \eqref{defL},  
       the minimization problem 
        in Problem~\ref{problem} with $\lambda=1$ becomes
        \begin{equation}\label{varProbphi2}
        \min\limits_{\substack{(\varphi,q)\in \mathcal{B}_{1}^p}}\mathcal{I}_2[\varphi,q]= \min\limits_{\substack{(\varphi,q)\in \mathcal{B}_{1}^p}} \int_{0}^{T}\!\!\!\int_{\mathbb{T}}
\big(L_2(q,\varphi_t,\varphi_x+1,\varphi_{xx})-V\varphi_x+G( \varphi_x+1)\big)\,\dx\dt.
        \end{equation}
        Next, we  prove the existence and uniqueness of  solutions to  this variational problem. We start by addressing the existence of solutions.        \begin{pro}\label{ex2}Suppose that Assumptions~\ref{assumtionOnG1}--\ref{assumtionOnH2}
hold  for some \(\beta>1\) and \(\gamma>1\) (see Section~\ref{assumptions}).
Then, the variational problem in \eqref{varProbphi2}  with \(p=\frac{\beta\gamma}{\beta+\gamma-1}
 \) has  a solution $(\varphi,q)\in \mathcal{B}_1^ \frac{\beta\gamma}{\beta+\gamma-1}$.
        \end{pro}
        \begin{proof}

The proof is very similar to that of Proposition~\ref{ex1}, for which reason we only highlight the main differences. 

Because $m_0$, $m_T\in C^2(\Tt)$,  there exists a positive constant, $\bar k$, such that
                \begin{equation}\label{bosecond}
                \Vert m_0\Vert_{C^1(\Tt)},\, \Vert m_T\Vert_{C^1(\Tt)}\leq \bar k.
                \end{equation}
Next, we observe                 that $\varphi^0$ in \eqref{phiTest} also belongs to $\mathcal{B}_1^\sigma$, where $\sigma=\frac{\beta\gamma}{\beta+\gamma-1}$, with
\begin{equation*}
\begin{aligned}
\varphi^0_{xx} = \frac{T-t}{T}(m_0)_x +\frac{t}{T}(m_T)_x.
\end{aligned}
\end{equation*}
Thus, we can find two positive constants, \(c\) and \(C\), only depending  on \(k_0\) (see \eqref{boabove}), \(\bar k\), \(L\), and \(G\), such that 
                \begin{equation}\label{funcBound2}
                -c\leq\inf\limits_{\substack{(\varphi,q)\in \mathcal{B}_1^\sigma}}\mathcal{I}_2[\varphi,q]\leq C. 
                \end{equation}
                
                Let $(\varphi^n,q_n)_{n\in\Nn}\subset\mathcal{B}_1^\sigma$ be a minimizing sequence for the variational problem in \eqref{varProbphi2} with \(p=\frac{\beta\gamma}{\beta+\gamma-1}
 \); that is,

                \begin{equation*}
                \lim\limits_{n\to\infty}\mathcal{I}_2[\varphi^n,q_n]=\inf\limits_{\substack{(\varphi,q)\in \mathcal{B}_1^\sigma}}\mathcal{I}_2[\varphi,q]=        \ell_{2}.
                \end{equation*}
                As in the proof of Proposition~\ref{ex1} (see  \eqref{seqboundphix} and \eqref{seqboundphit}), using  \eqref{funcBound2}  and the assumptions of the proposition, we conclude  that                \begin{equation}\label{seqboundphix2}
                \sup_{n\in\Nn}\int_{0}^{T}\!\!\!\int_{\Tt}|\varphi^n_x+1|^{\gamma}\,\dx\dt<+\infty
                \end{equation} 
                and
                \begin{equation}\label{seqboundphit2}
              \sup_{n\in\Nn}  \int_{0}^{T}\!\!\!\int_{\Tt}|\varphi^n_t+q_n-\varphi^n_{xx}|^\sigma\,\dx\dt< \infty.
                \end{equation}
                Moreover, similarly to \eqref{qLnorm}, we have from \eqref{seqboundphit2} and \eqref{derivph} that                \begin{equation}\label{boundq}
        \begin{split}
        \sup\limits_{n\in\Nn} \int_0^T |q_n|^\sigma\,\dt&=\sup\limits_{n\in\Nn}\int_{0}^{T}\left|
\int_{\Tt}q_n\,\dx\right|^\sigma\dt =\sup_{n\in\Nn}\int_{0}^{T}\left| \int_{\Tt}\varphi^n_t+q_n-\varphi^n_{xx}\,\dx\right| ^\sigma\dt \\ &\leq\sup_{n\in\Nn} \int_{0}^{T}\int_{\Tt}|\varphi^n_t+q_n-\varphi^n_{xx}|^\sigma\,\dx\dt<+\infty.
        \end{split}
        \end{equation}

                Next, we prove that the sequences \(\{\varphi^n_t\}_{n\in\Nn}\), \(\{\varphi^n_{xx}\}_{n\in\Nn}\), and \(\{\varphi^n\}_{n\in\Nn}\) are uniformly bounded with respect to the $L^\sigma$-norm. Using \eqref{boundq} in \eqref{seqboundphit2}, we get
                \begin{equation}\label{seqboundphit22}
                \sup_{n\in\Nn}  \int_{0}^{T}\!\!\!\int_{\Tt}|\varphi^n_t-\varphi^n_{xx}|^\sigma\,\dx\dt<+\infty.
                \end{equation}
                Set
                \begin{equation}\label{heateq}
                f(t,x)=\varphi^n_t(t,x)-\varphi^n_{xx}(t,x).
                \end{equation}
                Note that \eqref{heateq} is a non-homogeneous heat equation and, by \eqref{seqboundphit22}, we have $f\in L^\sigma([0,T]\times\Tt)$. Furthermore,  because $\varphi^n\in\mathcal{A}_1^\sigma$, we also have $\varphi^n(0,\cdot)\in C^3(\Tt)$. Thus, we can use the $L^p$-regularity theory for the heat equation (see  \cite[Theorem 5.4]{robinson2016N}), which yields
                \begin{equation}\label{seqboundphit-x}
                \sup\limits_{n\in\Nn}\left( \Vert\varphi^n_t\Vert_{L^\sigma([0,T]\times\Tt)}+\Vert\varphi^n_{xx}\Vert_{L^\sigma([0,T]\times\Tt)}\right) <\infty.
                \end{equation}
                Using the Poincar{\'e}--Wirtinger inequality, we conclude from  \eqref{seqboundphix2} and \eqref{seqboundphit-x} that                 \begin{equation}\label{boundphi_n2}
                \sup\limits_{n\in\Nn}   \Vert\varphi^n\Vert_{W^{1,\sigma}([0,T]\times\Tt)}<+\infty.
                \end{equation}
                
                Because $\gamma>\sigma>1$,   from \eqref{boundq}, \eqref{seqboundphit-x}, and \eqref{boundphi_n2}, we deduce that,  up to extracting  a subsequence if necessary, there exist  $\bar{q}\in L^\sigma([0,T])$ and  $\bar{\varphi}\in W^{1,\sigma}([0,T]\times\Tt)$ such that
\begin{equation}
\label{weaklyphiall2}
\begin{aligned}
& q_n\rightharpoonup\bar{q}\quad \text{weakly in} \,\ L^\sigma([0,T]),\\
& \varphi^n_{} \rightharpoonup\bar{\varphi}_{}\quad \text{weakly in}
\,\ W^{1,\sigma}([0,T]\times\Tt),\\ 
& \varphi^n_{xx} \rightharpoonup\bar{\varphi}_{xx}\quad
\text{weakly in} \,\ L^\sigma([0,T]\times\Tt).
\end{aligned}
\end{equation}

                Finally, relying on \eqref{weaklyphiall2} and arguing  as at the end of the proof of Proposition~\ref{exist1}, we conclude that $(\bar{\varphi},\bar{q})\in\mathcal{B}_1^\sigma$ and that it
                 minimizes the functional in \eqref{varProbphi2} over \(\mathcal{B}_1^\sigma\).
        \end{proof}
        
        Next, we prove the uniqueness of solutions to the variational problem in \eqref{varProbphi2}.
        
        \begin{pro}\label{uniq2} Suppose that Assumptions~\ref{assumtionOnG2}--\ref{assumtionOnH2}
hold  for some \(\beta>1\) and \(\gamma>1\) (see Section~\ref{assumptions}).
Then, the variational problem in \eqref{varProbphi2}  with \(p=\frac{\beta\gamma}{\beta+\gamma-1}
 \) has at most one solution.
        \end{pro}
        \begin{proof} The proof of this result  is identical to that   of Proposition~\ref{uniq1}, observing that   $\varphi^1_x=\varphi^2_x=\bar{\varphi}_{x}$ a.e.~in $[0,T]\times\Tt$ implies that $\varphi^1_{xx}=\varphi^2_{xx}=\bar{\varphi}_{xx}$ a.e.~in $[0,T]\times\Tt$. 
        \end{proof}
        Propositions \ref{ex2} and \ref{uniq2} yield the following theorem.
        \begin{teo}\label{exist11} Suppose that Assumptions~\ref{assumtionOnG2}--\ref{assumtionOnH2}
hold  for some \(\beta>1\) and \(\gamma>1\) (see Section~\ref{assumptions}).
Then, the variational problem in \eqref{varProbphi2}  with \(p=\frac{\beta\gamma}{\beta+\gamma-1}
 \) has a unique solution $(\varphi,q)\in \mathcal{B}_1^ \frac{\beta\gamma}{\beta+\gamma-1}$.
        \end{teo}
        \begin{proof}[Proof of Theorem~\ref{exist2}]The proof follows from Propositions~\ref{ex2} and \ref{uniq2}.
        \end{proof}

        \section{Weak Solution To  The First-Order Planning Problem With  Congestion}
\label{sec5}
In this section, we study the first-order planning problem with congestion stated in Problem~\ref{con}.  For convenience,  we  recall the associated PDE system here:
\begin{equation}\label{planningCS}
\begin{cases} 
-u_t+\frac{u_x^2}{2 m^\alpha}=m^\mu & \quad \text{in}\,\ (0,T)\times \Tt  \\ m_t- (u_x m^{1-\alpha})_x=0   & \quad \text{in}\,\ (0,T)\times \Tt  \\
m(0,x)=m_0(x)& \quad\text{in}\,\ \Tt\\
m(T,x)=m_T(x)& \quad\text{in}\,\ \Tt
\end{cases},
\end{equation}
where $\alpha>0$ and $\mu>0$.

We start by motivating the notion of the weak solutions to this
problem introduced in Definition~\ref{def-weak} below. Then, we prove a priori estimates for such weak solutions and establish their existence. 

\subsection{Motivation} Suppose  that
$(u,m)\in C^{2}([0,T]\times \Tt)\times C^{1}([0,T])$ solves \eqref{planningCS} with $m>0$. The second equation in \eqref{planningCS} can be written as
\begin{equation*}    
\div_{(t,x)}( m,-u_x m^{1-\alpha}) =0.
\end{equation*}
Accordingly, there exist functions $\varphi:(0,T)\times \Tt\to\Rr$  and $q:(0,T)\to\Rr$  such that
\begin{equation*}
\begin{cases}
m=\varphi_x+1   \\ u_x m^{1-\alpha}=\varphi_t+q.
\end{cases}
\end{equation*}
Thus, we have
\begin{equation}\label{ux1}
\begin{cases}
m=\varphi_x+1   \\ u_x =(\varphi_x+1)^{\alpha-1}(\varphi_t+q).
\end{cases}
\end{equation}
Differentiating the first equation in \eqref{planningCS} with respect to $x$, we get 
\begin{equation}\label{first}
-u_{xt}  +\frac{1}{2}\left(m^\alpha\left(  \frac{u_x}{m^{\alpha}}\right)^2 \right)_x=(m^\mu)_x.
\end{equation}
Substituting in \eqref{first}  the expressions for $u_x$ and $m$ obtained in \eqref{ux1}, and observing that \(\int_\Tt u_x\,\dx
=0\) by periodicity in the space variable, we get
\begin{equation}\label{eq: mot-}
\begin{cases}
F_1(\varphi,q)=0\\
F_2(\varphi,q)=0, 
\end{cases}
\end{equation}
where
\begin{equation}\label{def-F1F2}
\begin{split}
&F_1(\varphi,q)=- \left( \frac{\varphi_t+q}{ (\varphi_x+1)^{1-\alpha}}\right)_t  +\frac{1}{2}\left(\frac{(\varphi_t+q)^2}{ (\varphi_x+1)^{2-\alpha}} \right)_x- ((\varphi_x+1)^\mu)_x,\\
&F_2(\varphi,q)=\int_{\Tt} \frac{\varphi_t+q}{ (\varphi_x+1)^{1-\alpha}}\,\dx.
\end{split}
\end{equation}
Moreover, arguing as in Section~\ref{sec2} (see \eqref{phi_int_0}--\eqref{eq:iotas}),  we conclude that
\begin{equation}\label{first-pde'}
\begin{cases}
\displaystyle\int_{\Tt}\varphi(t,x)\,\dx=0, &t\in[0,T],
\\
\displaystyle\varphi(0,x)=\int_{0}^{x}\big(m_0(\tau)-1\big)\,\d \tau- \iota_0, &x\in\Tt,\\
\displaystyle\varphi(T,x)=\int_{0}^{x}\big(m_T(\tau)-1\big)\,\d
\tau- \iota_T, &x\in\Tt,
\end{cases}
\end{equation}
where $\iota_0$ and $\iota_T$ are the constants in \eqref{eq:iotas}.

We set 
\begin{equation*}
%\label{defbarApepsi}
\begin{split}
\mathcal{H}_0^+=\left\{\varphi\in C^2([0,T]\times\Tt): \,\varphi \text{ satisfies \eqref{first-pde'}, }\,\varphi_{x}+1>0\right\},
\end{split}
\end{equation*}
and    define the operator $A:D(A)\subset L^2([0,T]\times\Tt)\times L^2([0,T]) \to L^2([0,T]\times\Tt)\times L^2([0,T])$, with domain $D(A)=\mathcal{H}_0^+\times C^1([0,T])$,  by setting  
\begin{equation}\label{def_A}
A\begin{bmatrix}
\varphi \\q
\end{bmatrix} :=\begin{bmatrix}
F_1(\varphi,q)\\F_2(\varphi,q)
\end{bmatrix} \enspace\text{ for } (\varphi,q) \in D(A).
\end{equation}
Note   that, by \eqref{eq: mot-},  if   \((\varphi,q)\)  satisfies
 \eqref{ux1} with \((u,m)\)  as above, then
\begin{equation}\label{opA: 0}
A\begin{bmatrix}
\varphi \\q
\end{bmatrix} =\begin{bmatrix}
0 \\0
\end{bmatrix},
\end{equation}
which corresponds to the system of PDEs in Problem~\ref{con-var}.
As we state in the next proposition, given a solution of \eqref{opA: 0}, we can build a solution to \eqref{planningCS}. The proof is similar to that of Proposition \ref{prorel}, and so we omit it
here.

\begin{pro}\label{prorelc1}  Let  $(\varphi,q)\in C^{2}([0,T]\times \Tt)\times C^{1}([0,T])$ solve \eqref{opA: 0} in the classical sense. Assume that $\varphi_x+1>0$ and
        \begin{equation*}
        \begin{cases}
        m(t,x)=\varphi_x(t,x)+1   \\  u(t,x)=\int_{0}^{x}(\varphi_x(t,\tau)+1)^{\alpha-1}(\varphi_t(t,\tau)+q(t))\,\d \tau .
        \end{cases}
        \end{equation*}
Then, there exits a function, \(\vartheta:[0,T]\to\Rr\), only depending on \(t\), such that $(\tilde u, m)=(u + \vartheta,m)$ solves \eqref{planningCS}.
\end{pro}

To conclude this subsection, we prove that \(A\) is a monotone
operator. This property lies at the core of the notion of weak solution introduced in the following subsection.

\begin{pro}\label{mon-A}Assume that \(\alpha\in(0,2)\). Then,
the operator $A$ introduced in \eqref{def_A} is monotone in $L^2([0,T]\times\Tt)\times L^2([0,T])$; that is, 
        \begin{equation}\label{eq: monoA}
        \begin{split}
        \left( A\begin{bmatrix}
        \varphi \\q_1
        \end{bmatrix}-A\begin{bmatrix}
        \psi \\q_2
        \end{bmatrix}, \begin{bmatrix}
        \varphi \\q_1
        \end{bmatrix}-\begin{bmatrix}
        \psi \\q_2
        \end{bmatrix}\right)_{L^2\times L^2}\geq 0
        \end{split}
        \end{equation}
        for all $(\varphi,q_1)$, $(\psi,q_2) \in D(A)$.
\end{pro}
\begin{proof}
        Fix $(\varphi,q_1)$, $(\psi,q_2) \in   {{D(A)=\mathcal{H}_0^+\times C^1([0,T])}}$. 
        By   \eqref{def_A}, we have
        \begin{equation}\label{eq: mono}
        \begin{aligned}
        \left( A\begin{bmatrix}
        \varphi \\q_1
        \end{bmatrix}-A\begin{bmatrix}
        \psi \\q_2
        \end{bmatrix}, \begin{bmatrix}
        \varphi \\q_1
        \end{bmatrix}-\begin{bmatrix}
        \psi \\q_2
        \end{bmatrix}\right)_{L^2\times L^2}&=
      \int_{0}^{T}\!\!\!\int_{\Tt}\left( F_1(\varphi,q_1)-F_1(\psi,q_2)\right)(\varphi-\psi)\,\dx\dt\\&\quad+\int_{0}^{T}\!\!\!\int_{\Tt}\left( F_2(\varphi,q_1)-F_2(\psi,q_2)\right)(q_1-q_2)\,\dx\dt.
        \end{aligned}
        \end{equation}
        Integrating by parts, we obtain
        \begin{equation}\label{eq: monoF1}
        \begin{aligned}
        &\int_{0}^{T}\!\!\!\int_{\Tt}\left( F_1(\varphi,q_1)-F_1(\psi,q_2)\right)(\varphi-\psi)\,\dx\dt\\
 &\quad=\int_{0}^{T}\!\!\!\int_{\Tt}\left( \frac{\varphi_t+q_1}{ (\varphi_x+1)^{1-\alpha}}- \frac{\psi_t+q_2}{ (\psi_x+1)^{1-\alpha}}\right) (\varphi_t-\psi_t)\,\dx\dt\\&\quad\quad-\int_{0}^{T}\!\!\!\int_{\Tt}\frac{1}{2}\left(\frac{(\varphi_t+q_1)^2}{ (\varphi_x+1)^{2-\alpha}}-\frac{(\psi_t+q_2)^2}{ (\psi_x+1)^{2-\alpha}} \right)(\varphi_x-\psi_x)\,\dx\dt
\\&\quad\quad+\int_{0}^{T}\!\!\!\int_{\Tt} (\left( \varphi_x+1\right)^\mu-\left(
\psi_x+1\right)^\mu)(\varphi_x-\psi_x)\,\dx\dt.
        \end{aligned}
        \end{equation}
        Because $\mu>0$, $z\mapsto z^\mu$ is monotone increasing  in $\Rr_0^+$; thus,         \begin{equation}\label{term3}
        \left( \left( \varphi_x+1\right)^\mu-\left( \psi_x+1\right)^\mu\right) ((\varphi_x+1)-(\psi_x+1))\geq 0.
        \end{equation}
        %By using the preceding inequality and integration by parts, we deduce
        %       \begin{equation}\label{term3}
        %\begin{split}
        %-&\int_{0}^{T}\int_{\mathbb{T}}  (\left( \varphi_x+1\right)^\mu-\left( \psi_x+1\right)^\mu)_x(\varphi-\psi)\dx\dt\\=&\int_{0}^{T}\int_{\mathbb{T}}  (\left( \varphi_x+1\right)^\mu-\left( \psi_x+1\right)^\mu)((\varphi_x+1)-(\psi_x+1))\dx\dt\geq0.
        %\end{split}
        %\end{equation}
        Recalling \eqref{def-F1F2},  we combine the first term on the right-hand side of \eqref{eq: monoF1} with the second term on the right-hand side of \eqref{eq: mono} to get
        \begin{equation}\label{id: 1}
        \begin{split}
        &\int_{0}^{T}\!\!\!\int_{\Tt}\bigg[ \bigg( \frac{\varphi_t+q_1}{ (\varphi_x+1)^{1-\alpha}}- \frac{\psi_t+q_2}{ (\psi_x+1)^{1-\alpha}}\bigg) (\varphi_t-\psi_t)  +\left( F_2(\varphi,q_1)-F_2(\psi,q_2)\right)(q_1-q_2)\bigg]\,\dx\dt\\&\quad=\int_{0}^{T}\!\!\!\int_{\mathbb{T}} \left( \frac{\varphi_t+q_1}{ (\varphi_x+1)^{1-\alpha}}- \frac{\psi_t+q_2}{ (\psi_x+1)^{1-\alpha}}\right) ((\varphi_t+q_1)-(\psi_t+q_2))\,\dx\dt.
        \end{split}
        \end{equation}
        %Again integrating the  by parts the third integrand of the right-hand side of \eqref{eq: mono}, we get
        %       \begin{equation}\label{id: 2}
        %\begin{split}
        %&\int_{0}^{T}\int_{\Tt}\frac{1}{2}\left(\frac{(\varphi_t+q_1)^2}{ (\varphi_x+1)^{2-\alpha}}-\frac{(\psi_t+q_2)^2}{ (\psi_x+1)^{2-\alpha}} \right)_x(\varphi-\psi) \dx\dt\\=&-\int_{0}^{T}\int_{\Tt}\frac{1}{2}\left(\frac{(\varphi_t+q_1)^2}{ (\varphi_x+1)^{2-\alpha}}-\frac{(\psi_t+q_2)^2}{ (\psi_x+1)^{2-\alpha}} \right)((\varphi_x+1)-(\psi_x+1)) \dx\dt.
        %\end{split}
        %\end{equation}
        Using \eqref{id: 1} and \eqref{eq: monoF1} in \eqref{eq: mono} and taking into account \eqref{term3}, we obtain
that        \begin{equation}\label{eq: mono2}
        \begin{split}
        &\left( A\begin{bmatrix}
        \varphi \\q_1
        \end{bmatrix}-A\begin{bmatrix}
        \psi \\q_2
        \end{bmatrix}, \begin{bmatrix}
        \varphi \\q_1
        \end{bmatrix}-\begin{bmatrix}
        \psi \\q_2
        \end{bmatrix}\right)\\&\quad\geq\int_{0}^{T}\!\!\!\int_{\Tt} \left( \frac{\varphi_t+q_1}{ (\varphi_x+1)^{1-\alpha}}- \frac{\psi_t+q_2}{ (\psi_x+1)^{1-\alpha}}\right) ((\varphi_t+q_1)-(\psi_t+q_2))\,\dx\dt\\&\qquad-\int_{0}^{T}\!\!\!\int_{\Tt}\frac{1}{2}\left(\frac{(\varphi_t+q_1)^2}{ (\varphi_x+1)^{2-\alpha}}-\frac{(\psi_t+q_2)^2}{ (\psi_x+1)^{2-\alpha}} \right)((\varphi_x+1)-(\psi_x+1)) \dx\dt.
        \end{split}
        \end{equation}
        To prove that the right-hand side of \eqref{eq: mono2}
        is nonnegative, which yields  \eqref{eq: monoA},  it suffices to show that
        \begin{equation}\label{eq: Ca}
        \begin{split}
        0&\leq\left( \frac{\varphi_t+q_1}{ (\varphi_x+1)^{1-\alpha}}- \frac{\psi_t+q_2}{ (\psi_x+1)^{1-\alpha}}\right) ((\varphi_t+q_1)-(\psi_t+q_2))\\&\quad-\frac{1}{2}\left(\frac{(\varphi_t+q_1)^2}{ (\varphi_x+1)^{2-\alpha}}-\frac{(\psi_t+q_2)^2}{ (\psi_x+1)^{2-\alpha}} \right)((\varphi_x+1)-(\psi_x+1))\\&=\left( \frac{1}{2}+ \frac{\psi_x+1}{2(\varphi_x+1)}\right)\frac{(\varphi_t+q_1)^2}{ (\varphi_x+1)^{1-\alpha}}+\left( \frac{1}{2}+ \frac{\varphi_x+1}{2(\psi_x+1)}\right)\frac{(\psi_t+q_2)^2}{ (\psi_x+1)^{1-\alpha}}\\&\quad-\frac{(\varphi_t+q_1)(\psi_t+q_2)}{ (\varphi_x+1)^{1-\alpha}}- \frac{(\psi_t+q_2)(\varphi_t+q_1)}{ (\psi_x+1)^{1-\alpha}}
        \end{split}
        \end{equation}
        in $[0,T]\times\Tt$. Because $\varphi_x+1>0$ and $\psi_x+1>0$
in $[0,T]\times\Tt$, we have that \eqref{eq: Ca} holds when $\varphi_t+q_1$ and $\psi_t+q_2$ have different signs. Thus, to complete the proof, we are left to prove  \eqref{eq: Ca} in the case where $\varphi_t+q_1$ and $\psi_t+q_2$ have the same sign.

Assume that  $(\varphi_t+q_1)(\psi_t+q_2)\geq
0$. Note that the  Cauchy inequality implies that
        \begin{equation}\label{Cauchy}
        \begin{split}
        &\left( \left( \frac{1}{2}+ \frac{\psi_x+1}{2(\varphi_x+1)}\right)\frac{(\varphi_t+q_1)^2}{ (\varphi_x+1)^{1-\alpha}}+\left( \frac{1}{2}+ \frac{\varphi_x+1}{2(\psi_x+1)}\right)\frac{(\psi_t+q_2)^2}{ (\psi_x+1)^{1-\alpha}}\right) ^2\\ &\quad\geq 4\left( \frac{1}{2}+ \frac{\psi_x+1}{2(\varphi_x+1)}\right)\left( \frac{1}{2}+ \frac{\varphi_x+1}{2(\psi_x+1)}\right)\frac{(\varphi_t+q_1)^2(\psi_t+q_2)^2}{(\varphi_x+1)^{1-\alpha}(\psi_x+1)^{1-\alpha}}.
        \end{split}
        \end{equation}
        %Next, we prove 
        %\begin{equation}\label{pos2}
        %2\left( \frac{1}{2}+ \frac{\psi_x+1}{2(\varphi_x+1)}\right)^{\frac{1}{2}}\left( \frac{1}{2}+ \frac{\varphi_x+1}{2(\psi_x+1)}\right)^{\frac{1}{2}}\frac{1}{(\varphi_x+1)^{\frac{1-\alpha}{2}}(\psi_x+1)^{\frac{1-\alpha}{2}}}-\left( \frac{1}{(\varphi_x+1)^{1-\alpha}}+ \frac{1}{(\psi_x+1)^{1-\alpha}}\right)\geq 0. 
        %\end{equation}
        Taking into account that $\varphi_x+1>0$ and $\psi_x+1>0$, and  recalling that $0<\alpha<2$, we get 
        \begin{equation*}%\label{pos3}
        \begin{split}
        &4\left( \frac{1}{2}+ \frac{\psi_x+1}{2(\varphi_x+1)}\right)\left( \frac{1}{2}+ \frac{\varphi_x+1}{2(\psi_x+1)}\right)\\&\quad-(\varphi_x+1)^{1-\alpha}(\psi_x+1)^{1-\alpha}\left( \frac{1}{(\varphi_x+1)^{1-\alpha}}+ \frac{1}{(\psi_x+1)^{1-\alpha}}\right)^2\\=\,\,&
        2+\frac{\varphi_x+1}{\psi_x+1}+\frac{\psi_x+1}{\varphi_x+1}-\left(\left( \frac{\varphi_x+1}{\psi_x+1}\right) ^{1-\alpha}+2+\left( \frac{\psi_x+1}{\varphi_x+1}\right) ^{1-\alpha}\right)\\=\,\,&\frac{(\varphi_x+1)^{2-\alpha}-(\psi_x+1)^{2-\alpha}}{(\psi_x+1)(\varphi_x+1)^{1-\alpha}}+ \frac{(\psi_x+1)^{2-\alpha}-(\varphi_x+1)^{2-\alpha}}{(\varphi_x+1)(\psi_x+1)^{1-\alpha}}\\=\,\,&\frac{1}{(\varphi_x+1)(\psi_x+1)}\left( (\varphi_x+1)^{2-\alpha}-(\psi_x+1)^{2-\alpha}\right) \left( (\varphi_x+1)^{\alpha}-(\psi_x+1)^{\alpha}\right)\geq 0.
        \end{split}
        \end{equation*}
Consequently, 
\begin{equation*}
\begin{aligned}
&4\left( \frac{1}{2}+
\frac{\psi_x+1}{2(\varphi_x+1)}\right)\left( \frac{1}{2}+ \frac{\varphi_x+1}
{2(\psi_x+1)}\right)\frac{(\varphi_t+q_1)^2(\psi_t+q_2)^2}
{(\varphi_x+1)^{1-\alpha}(\psi_x+1)^{1-\alpha}} \\
&\quad\geq \left(
\frac{1}{(\varphi_x+1)^{1-\alpha}}+ \frac{1}{(\psi_x+1)^{1-\alpha}}\right)^2
(\varphi_t+q_1)^2(\psi_t+q_2)^2. 
\end{aligned}
\end{equation*}      
Using the preceding estimate in \eqref{Cauchy} first, and then   taking the square root in the resulting estimate recalling 
that  $(\varphi_t+q_1)(\psi_t+q_2)\geq
0$, we obtain%
\begin{equation*}
\begin{aligned}
&\left( \frac{1}{2}+ \frac{\psi_x+1}{2(\varphi_x+1)}\right)\frac{(\varphi_t+q_1)^2}{
(\varphi_x+1)^{1-\alpha}}+\left( \frac{1}{2}+ \frac{\varphi_x+1}{2(\psi_x+1)}\right)\frac{(\psi_t+q_2)^2}{
(\psi_x+1)^{1-\alpha}}\\&\quad\geq\frac{(\varphi_t+q_1)(\psi_t+q_2)}{
(\varphi_x+1)^{1-\alpha}}+ \frac{(\psi_t+q_2)(\varphi_t+q_1)}{
(\psi_x+1)^{1-\alpha}}.
\end{aligned}
\end{equation*}
This completes the proof of \eqref{eq: Ca}.
\end{proof}

\subsection{Weak Solutions}
Next, based on the monotonicity of the operator \(A\) proved
in Proposition~\ref{mon-A}, we  introduce a notion of  weak solutions for the system \eqref{opA: 0}
and prove their existence. This notion  mimics the notion of  solutions to the weak variational inequality associated with a monotone operator (see \cite{KiSt00}).  This approach based on monotonicity
 for solving MFGs  was introduced in \cite{FG2} and further developed in \cite{FGT1,FeGoTa20}.

 \begin{definition}\label{def-weak} We say that a pair $(\varphi,q)$
 is a weak solution to \eqref{opA: 0} (or, equivalently,  to Problem~\ref{con-var})  if there is $\kappa>1$ such that $(\varphi,q)\in \mathcal{A}_0^{\kappa}\times
L^\kappa([0,T])$  (see \eqref{def-A_0})\ and, for
all $(\psi,\varpi)\in D(A)$, we   have
        \begin{equation}\label{def--weak-sol}
        \left\langle A\begin{bmatrix}
        \psi \\ \varpi
        \end{bmatrix}
        , \begin{bmatrix}
        \psi \\ \varpi
        \end{bmatrix}-\begin{bmatrix}
        \varphi \\q
        \end{bmatrix}\right\rangle_{L^{\kappa'}\times L^{\kappa'},
L^\kappa \times L^k}\geq 0.
        \end{equation}
        \end{definition}

Next, we prove a priori estimates for the classical solutions of the system \eqref{opA: 0}, which we later extend to weak solutions.

\begin{proposition} \label{apri-1}Let  $A$ be given by \eqref{def_A}. Suppose that Assumption~\ref{assumtionOnBoundsA} holds and that $(\varphi,q)\in
 D(A)$ solves \eqref{opA: 0} with $\alpha<\mu+1$. Then, there exists a positive constant, $C$, independent  of \((\varphi,q)\), such that
        \begin{equation*}
        \int_{0}^{T}\!\!\!\int_{\mathbb{T}} ( \varphi_x+1)^{\mu+1}\,\dx\dt\leq C.
        \end{equation*}
Moreover, for \(\alpha\in(0,2)\),
\begin{align*}
&\int_{0}^{T}\!\!\!\int_{\mathbb{T}}\Big(|\varphi_t|^\frac{2(\mu+1)}{\mu+2-\alpha}+|q|^\frac{2(\mu+1)}{\mu+2-\alpha}\Big)\,\dx\dt\leq
C \qquad \text{if } \alpha\in(0,1],\\
& \int_{0}^{T}\!\!\!\int_{\mathbb{T}}\Big(|\varphi_t|^\frac{2(\mu+1)}{\mu+3-\alpha}+|q|^\frac{2(\mu+1)}{\mu+3-\alpha}\Big)\,\dx\dt\leq
C \qquad \text{if } \alpha\in[1,2).
\end{align*}
       \end{proposition}
       
\begin{proof} Because $(\varphi,q)$ solves \eqref{opA: 0} pointwise, we have
        \begin{equation}\label{eq: energy}
        \begin{split}
        \left( A\begin{bmatrix}
        \varphi \\q     
        \end{bmatrix}, \begin{bmatrix}
        \varphi \\q
        \end{bmatrix}\right)_{L^2\times L^2} =  \left( A\begin{bmatrix}
        \varphi \\q     
        \end{bmatrix}, \begin{bmatrix}
        \varphi^0 \\q^0
        \end{bmatrix}\right)_{L^2\times L^2} 
        \end{split},
        \end{equation}
        where $q^0=0$ and $\varphi^0$ defined by \eqref{phiTest}.
        Recalling \eqref{def-F1F2} and \eqref{def_A}, an integration by parts yields           \begin{equation}\label{eq-equal1}
        \begin{split}
        \left( A\begin{bmatrix}
        \varphi \\q     
        \end{bmatrix}, \begin{bmatrix}
        \varphi \\q
        \end{bmatrix}\right)_{L^2\times L^2} =-&\int_{\Tt} \frac{\varphi_t+q}{ (\varphi_x+1)^{1-\alpha}}\varphi\,\dx\bigg|_0^T +\int_{0}^{T}\!\!\!\int_{\Tt}\frac{(\varphi_t+q)^2}{ (\varphi_x+1)^{1-\alpha}}\,\dx\dt\\-&\frac{1}{2}\int_{0}^{T}\!\!\!\int_{\Tt} \frac{(\varphi_t+q)^2}{ (\varphi_x+1)^{2-\alpha}}\varphi_x
\,\dx\dt +\int_{0}^{T}\!\!\!\int_{\Tt}(\varphi_x+1)^{\mu}\varphi_x\,\dx\dt\\=-&\int_{\Tt} \frac{\varphi_t+q}{ (\varphi_x+1)^{1-\alpha}}\varphi\,\dx\bigg|_0^T +\frac{1}{2}\int_{0}^{T}\!\!\!\int_{\Tt}\frac{(\varphi_t+q)^2}{ (\varphi_x+1)^{1-\alpha}}\dx\dt\\+&\frac{1}{2}\int_{0}^{T}\!\!\!\int_{\Tt} \frac{(\varphi_t+q)^2}{ (\varphi_x+1)^{2-\alpha}}\,\dx\dt+\int_{0}^{T}\!\!\!\int_{\Tt}
\big((\varphi_x+1)^{\mu+1}-(\varphi_x+1)^{\mu}\big)\,\dx\dt.
        \end{split}
        \end{equation}
        Because $q^0=0$,  we similarly obtain that
               \begin{equation}\label{eq-equal2}
        \begin{split}
        \left( A\begin{bmatrix}
        \varphi \\q     
        \end{bmatrix}, \begin{bmatrix}
        \varphi^0 \\q^0
        \end{bmatrix}\right)_{L^2\times L^2} =-&\int_{\Tt} \frac{\varphi_t+q}{ (\varphi_x+1)^{1-\alpha}}\varphi^0\,\dx\bigg|_0^T +\int_{0}^{T}\!\!\!\int_{\Tt}\frac{(\varphi_t+q)\varphi^0_t}{ (\varphi_x+1)^{1-\alpha}}\,\dx\dt\\-&\frac{1}{2}\int_{0}^{T}\!\!\!\int_{\Tt}\frac{(\varphi_t+q)^2}{ (\varphi_x+1)^{2-\alpha}}(\varphi^0_x+1)\,\dx\dt +\frac{1}{2}\int_{0}^{T}\!\!\!\int_{\Tt}\frac{(\varphi_t+q)^2}{ (\varphi_x+1)^{2-\alpha}}\,\dx\dt\\+&\int_{0}^{T}\!\!\!\int_{\Tt}\big((\varphi_x+1)^{\mu}(\varphi^0_x+1)-(\varphi_x+1)^{\mu}\big)\,\dx\dt.
        \end{split}
        \end{equation}
        Then, from \eqref{eq-equal1},  \eqref{eq-equal2}, and \eqref{eq: energy}, together with the fact that \(\varphi - \varphi^0=0\)
on \(\{0,T\}\times\Tt\), we get
        \begin{equation*}
        \begin{split}
        &\int_{0}^{T}\!\!\!\int_{\Tt}\bigg(\frac{1}{2} \frac{(\varphi_t+q)^2}{ (\varphi_x+1)^{1-\alpha}}+\frac{1}{2}\frac{(\varphi_t+q)^2}{ (\varphi_x+1)^{2-\alpha}}(\varphi^0_x+1) +(\varphi_x+1)^{\mu+1}\bigg)\,\dx\dt\\ &\quad= \int_{0}^{T}\!\!\!\int_{\Tt}\bigg(\frac{(\varphi_t+q)\varphi^0_t }{ (\varphi_x+1)^{\frac{2-\alpha}{2}}(\varphi_x+1)^{-\frac{\alpha}{2}}}+(\varphi_x+1)^{\mu}(\varphi^0_x+1)\bigg)\,\dx\dt.
        \end{split}
        \end{equation*}
        Next,  using \eqref{eq:bddffixt} (in view of Assumption \ref{assumtionOnBoundsA}), Young's inequality, and recalling that $\alpha< \mu+1$, we deduce that
        \begin{equation*}
        \begin{split}
        &\int_{0}^{T}\!\!\!\int_{\Tt} \bigg( \frac{1}{2}\frac{(\varphi_t+q)^2}{ (\varphi_x+1)^{1-\alpha}}+\frac{k_0}{2}\frac{(\varphi_t+q)^2}{ (\varphi_x+1)^{2-\alpha}}+(\varphi_x+1)^{\mu+1}\bigg)\,\dx\dt\\&\quad\leq
        \int_{0}^{T}\!\!\!\int_{\Tt}\bigg(\frac{k_0}{4}\frac{(\varphi_t+q)^2}{ (\varphi_x+1)^{2-\alpha}}+\frac{|\varphi^0_{t}|^2}{k_0}(\varphi_x+1)^{\alpha}+k_1(\varphi_x+1)^{\mu}\bigg)\,\dx\dt\\&\quad\leq
        \int_{0}^{T}\!\!\!\int_{\Tt}\bigg(\frac{k_0}{4}\frac{(\varphi_t+q)^2}{ (\varphi_x+1)^{2-\alpha}}+\frac{1}{2}(\varphi_x+1)^{\mu+1}+C\bigg)\,\dx\dt,
        \end{split}
        \end{equation*}
        where \(C\) depends only on \(\alpha\), \(\mu\), \(T\),
        \(k_0\), and \(k_1\). Thus, we have
        \begin{equation*}
        \begin{split}
        \int_{0}^{T}\!\!\!\int_{\Tt}\bigg(\frac{(\varphi_t+q)^2}{ (\varphi_x+1)^{1-\alpha}}+\frac{(\varphi_t+q)^2}{ (\varphi_x+1)^{2-\alpha}}+(\varphi_x+1)^{\mu+1}\bigg)\,\dx\dt\leq C.
        \end{split}
        \end{equation*}
        Finally,  arguing as in the proof of Proposition~\ref{ex1}
to obtain  \eqref{qLnorm} and \eqref{phitrnorm}, with  \(\sigma=\frac{2(\mu+1)}{\mu
+2 -\alpha}\), \(\gamma=\mu+1\),  \(\beta= 2-\alpha\) if \(\alpha\in
(0,1]\), and \(\beta= 3-\alpha\) if \(\alpha\in
(1,2)\),  concludes the proof of Proposition~\ref{apri-1}.
\end{proof}

To prove the existence of solutions for the system \eqref{opA: 0} of non-linear equations, we consider a regularized problem.
More precisely,   we fix  \(\epsi>0\) and study the following regularization of \eqref{eq: mot-}--\eqref{first-pde'}:
\begin{equation}\label{eq: mot-2}
\begin{cases}
\varepsilon\big(\varphi+\sum_{|j|=6}\partial^{2j}\varphi\big)+F_1(\varphi,q)=0,&\quad (t,x)\in [0,T]\times\Tt,\\[.7mm]
\varepsilon (q-q^{\prime\prime})+F_2(\varphi,q)=0,&\quad t\in [0,T],
\\[.7mm]
\int_{\Tt}\varphi(t,x)\,\dx=0,&\quad t\in [0,T],
\end{cases}
\end{equation}
with boundary conditions 
\begin{equation}\label{eq: mot-2-bound}
\begin{cases}
\sum_{k=i}^{6}\frac{\partial^{2k-i}}{\partial t^{2k-i}}\frac{\partial^{2(6-k)}}{\partial x^{2(6-k)}}\varphi=0, & (t,x)\in\{0,T\}\times\Tt,\enspace  i=2,\dots, 6,\\[.7mm]
\varphi(0,x)=\int_{0}^{x}\big(m_0(\tau)-1\big)\,\d \tau- \iota_0, &x\in\Tt, \\[.7mm] \varphi(T,x)=\int_{0}^{x}\big(m_T(\tau)-1\big)\,\d \tau- \iota_T, & x\in\Tt,\\
q^\prime(0)=q^\prime(T)=0.
\end{cases}.
\end{equation}
 
In \eqref{eq: mot-2}, the differential operator $\partial^{2j}$ is   with respect
to the pair \((t,x)\); that is, $\partial^{2j}\psi=\frac{\partial^{|2j|}\psi}{\partial t^{2j_0} \partial x^{2j_1}}$ for  $ j=(j_0,j_1) \in \Nn_0 \times \Nn_0$. 

Taking into account the first equation in \eqref{ux1}, we are interested in solutions of \eqref{eq: mot-2} that satisfy $\varphi_x+1>0$. To do so, we first introduce the set 
\begin{equation}
\label{eq:Mepsi}
\begin{aligned}
\mathcal{M}^{k}_{\varepsilon} = \bigg\{&\varphi\in H^{k}([0,T]\times\Tt)
:\, \varphi_x+1\geq \epsi;\,\, \int_{\Tt}\varphi(t,x)\,\dx=0;
\\  & \quad\varphi(0,x)=\int_{0}^{x}\big(m_0(\tau)-1\big)\,\d
\tau- \iota_0;\,\,\varphi(T,x)=\int_{0}^{x}\big(m_T(\tau)-1\big)\,\d
\tau- \iota_T\bigg\}, 
\end{aligned}
\end{equation}
where $\iota_0$ and $\iota_T$ are the constants in \eqref{eq:iotas}
and \(k>1\); furthermore,  note that \(\mathcal{M}^{k}_{\varepsilon} \subset
C^2([0,T]\times\Tt)\) provided \(k\geq 4\) due to 
Morrey's embedding  theorem.
Then, we build an operator that to each $(\varphi_0,q_0)\in \mathcal{M}^{5}_{\varepsilon}\times
H^{1}([0,T])$ maps the solution \((\varphi,q)\) of the system
\begin{equation}\label{EL}
\begin{cases}
\varepsilon(\varphi+\sum_{|j|=6}\partial^{2j}\varphi)+F_1(\varphi_0,q_0)=0,&  (t,x)\in [0,T]\times\Tt,\\[1.4mm]
\sum_{j=i}^{6}\frac{\partial^{2j-i}}{\partial t^{2j-i}}\frac{\partial^{2(6-j)}}{\partial
x^{2(6-j)}}\varphi=0,& (t,x)\in\{0,T\}\times\Tt,\enspace  i=2,\dots, 6,\\[1.1mm]
\varepsilon (q-q^{\prime\prime})+F_2(\varphi_0,q_0)=0,&  t\in [0,T],\\
q^\prime(0)=q^\prime(T)=0.
\end{cases}
\end{equation}
 
A key observation is that the fixed points of this operator solve \eqref{eq: mot-2}. We further note that the first two identities in \eqref{EL} comprise the system of Euler--Lagrange equations of the functional 
\begin{equation}\label{funct-vp}
I_{(\varphi_0,q_0)}[\varphi]=\int_{0}^{T}\!\!\!\int_{\Tt}\bigg[\frac{\varepsilon}{2}\bigg( \varphi^2+\sum_{|j|=6}^{}(\partial^{j}\varphi)^2\bigg)+ F_1(\varphi_0,q_0)\varphi\bigg]\,\dx\dt. 
\end{equation}
To impose the condition $\varphi_{x}+1>0$, we consider the variational problem of finding  \(\bar \varphi\in\mathcal{M}^{6}_{\varepsilon} \) such that
\begin{equation}\label{vpp}
I_{(\varphi_0,q_0)}[\bar \varphi]=\inf\limits_{\varphi\in \mathcal{M}^{6}_{\varepsilon}} I_{(\varphi_0,q_0)}[\varphi], 
\end{equation}
where \(I_{(\varphi_0,q_0)}[\cdot]\) is the functional in \eqref{funct-vp}
and \(\mathcal{M}^{6}_{\varepsilon}\) is the set in \eqref{eq:Mepsi}
with \(k=6\).

\begin{pro}\label{prop:minpb} Assume that Assumption~\ref{assumtionOnBoundsA} holds and that \(m_0\), \(m_T \in H^{5}(\Tt)\). Fix  $(\varphi_0,q_0)\in \mathcal{M}^{5}_{\varepsilon}\times H^{1}([0,T])$ with \(0<\epsi\leq k_0\). Then,  the variational problem \eqref{vpp}
        has a unique solution,   \(\bar \varphi\in\mathcal{M}^{6}_{\varepsilon}
\).
Moreover, there exists a positive constant, \(C>0\), depending only on the problem data and on  $\varphi_0$,  $q_0$, and \(\epsi\),
 such that 
        \begin{equation}\label{m-bo}
        ||\bar \varphi||_{H^{6}([0,T]\times\Tt)}\leq C.
        \end{equation}\end{pro}
\begin{proof} First, we prove that the  infimum in \eqref{vpp},
which we denote by \(I_0\), is finite.  Note that the function  $\varphi^0$   in \eqref{phiTest} satisfies  $\varphi^0\in\mathcal{M}^{6}_\varepsilon$ for  \(0<\epsi\leq k_0\). Because $(\varphi_0,q_0)\in \mathcal{M}^{5}_\varepsilon\times H^1([0,T])\subset C^2([0,T]\times\Tt) \times C([0,T])$, we find that   
        \begin{equation}\label{uper}
        I_0\leq I_{(\varphi_0,q_0)}[\varphi^0]\leq C^{\varepsilon}_0 \in \Rr^+.
        \end{equation}
        
On the other hand, using 
Young's inequality, we conclude that for all \(\varphi\in \mathcal{M}^{6}_\epsi\), we have
\begin{equation}
\begin{aligned}\label{lowerms}
I_{(\varphi_0,q_0)}[\varphi] &\geq \int_{0}^{T}\!\!\!\int_{\Tt}\bigg[\frac{\varepsilon}{2}\bigg(
\varphi^2+\sum_{|j|=6}^{}(\partial^{j}\varphi)^2\bigg)- \frac{1}{\epsi}(F_1(\varphi_0,q_0))^2 - \frac{\epsi}{4}\varphi^2\bigg]\,\dx\dt \\ &\geq -\int_{0}^{T} \!\!\!     \int_{\mathbb{T}}\frac{1}{\varepsilon}\left( F_1(\varphi_0,q_0)\right)
^2\,\dx\dt =-C^{\varepsilon}_1 \in\Rr^-_0. 
\end{aligned}
\end{equation}
Hence, taking the infimum over all 
 \(\varphi\in \mathcal{M}^{6}_\epsi\), we conclude that \(I_0 \geq -C^{\varepsilon}_1\). Thus, $I_0$ is finite.

        Let $\{\varphi_n\}_{n=1}^\infty\subset \mathcal{M}^{6}_\varepsilon$ be
a minimizing sequence for \eqref{vpp}; that is, $\{\varphi_n\}_{n=1}^\infty\subset
\mathcal{M}^{6}_\varepsilon$  such that
\begin{equation*}
\begin{aligned}
I_0=\lim_n I_{(\varphi_0,q_0)}[\varphi_n].
\end{aligned}
\end{equation*}
Then, using \eqref{uper} and the first estimate in \eqref{lowerms},  there exists $n_0\in \Nn$ such that, for all $n\geq n_0$, we have 
        \begin{equation}\label{ep-bound}
        \begin{aligned}
        \int_{0}^{T}\!\!\!\int_{\Tt}\bigg(\frac{\varepsilon}{4} \varphi_n^2+ \frac{\varepsilon}{2}\sum_{|j|=6}^{}(\partial^{j}\varphi_n)^2\bigg)\,\dx\dt&\leq I_{(\varphi_0,q_0)}[\varphi_n] +  \int_{0}^{T}\!\!\!\int_{\Tt} \frac{1}{\epsi}(F_1(\varphi_0,q_0))^2\,\dx\dt\\ &\leq I_{0}+1+  \int_{0}^{T}\!\!\!\int_{\Tt} \frac{1}{\epsi}(F_1(\varphi_0,q_0))^2\,\dx\dt \\&\leq \tilde C_0^\epsi \in \Rr^+. 
\end{aligned}
        \end{equation} 
        Furthermore, by the  Gagliardo--Nirenberg inequality,  for any multi-index \(j\in\Nn_0\times\Nn_0\) with  $|j|\leq 6$, we have 
        \begin{equation}\label{GGG}
        \Vert \partial^j\varphi_n\Vert ^2_{L^2([0,T]\times\Tt)}\leq C\left(\Vert \varphi_n\Vert ^2_{L^2([0,T]\times\Tt)} +\Vert D^{6}\varphi_n\Vert ^2_{L^2([0,T]\times\Tt)}\right). 
        \end{equation}
        The preceding estimates and \eqref{ep-bound} yield that $\{\varphi_n\}_{n=1}^\infty$ is bounded in $H^{6}([0,T]\times\Tt)$. Hence, there exist \(\bar \varphi\in H^{6}([0,T]\times\Tt) \) and a  subsequence, \(\{\varphi_{n_k}\}_{k=1}^\infty\), of \(\{\varphi_n\}_{n=1}^\infty\) such that
 $\varphi_{n_k}\rightharpoonup \bar\varphi$ in $H^{6}([0,T]\times\Tt)$. Moreover,
by the Rellich--Kondrachov theorem, $\varphi_{n_k}\to \bar\varphi$ in $H^{5}([0,T]\times\Tt)$. Using, in addition, the continuity of the trace, we conclude that \(\bar\varphi\in \mathcal{M}^{6}_\varepsilon\). Moreover, using  the preceding convergences and the lower-semicontinuity of the \(L^2\)-norm with respect to the weak convergence in \(L^2\), we deduce that
         \begin{equation*}
        I_0\leq I_{(\varphi_0,q_0)}[\bar\varphi]\leq \liminf\limits_{k\to\infty}I_{(\varphi_0,q_0)}[\varphi_{n_k}]=\lim_n I_{(\varphi_0,q_0)}[\varphi_n]= I_0.
        \end{equation*} 
        Thus, $\bar \varphi$ is a minimizer of \(I_{(\varphi_0,q_0)}[\cdot]\) over \(\mathcal{M}^{6}_\epsi\).  We observe further  that \eqref{m-bo}
follows from   \eqref{ep-bound} and \eqref{GGG}.
        
        Now, we prove the uniqueness of the minimizer. Suppose that $\varphi$, $\tilde{\varphi}\in \mathcal{M}^{6}_\varepsilon$ solve \eqref{vpp} with $\varphi\neq \tilde{\varphi}$. Note that $\frac{\varphi+\tilde{\varphi}}{2}\in\mathcal{M}^{6}_\varepsilon$ and $\varphi-\tilde{\varphi}\in C([0,T]\times\Tt)$ with
        \begin{equation*}
        \int_{0}^{T}\!\!\!\int_{\Tt} (\varphi-\tilde{\varphi})^2\,\dx\dt>0.
        \end{equation*}
        Then, 
        \begin{equation}\label{min-main-ineq}
        \begin{aligned}
        I_0&\leq I_{(\varphi_0,q_0)}\left[ \frac{\varphi+\tilde{\varphi}}{2}\right] \\&=\int_{0}^{T}\!\!\!\int_{\Tt} \bigg[\frac{\varepsilon}{2}\bigg( \bigg( \frac{\varphi+\tilde{\varphi}}{2}\bigg)^2 +\sum_{|j|=6}^{}\bigg( \frac{\partial^j\varphi+\partial^j\tilde{\varphi}}{2}\bigg)^2\bigg)  +F_1(\varphi_0,q_0)\frac{\varphi+\tilde{\varphi}}{2} \bigg] \,\dx\dt\\
 &=\frac{1}{2}\int_{0}^{T}\!\!\!\int_{\Tt}\bigg[\frac{\varepsilon}{2}\bigg(  \varphi^2 +\sum_{|j|=6}^{}\left( \partial^j\varphi\right)^2\bigg)  +F_1(\varphi_0,q_0)\varphi \bigg]\,\dx\dt\\ &\quad+\frac{1}{2}\int_{0}^{T}\!\!\!\int_{\Tt} \bigg[\frac{\varepsilon}{2}\bigg(  \tilde{\varphi}^2 +\sum_{|j|=6}^{}\left( \partial^j\tilde{\varphi}\right)^2\bigg)+F_1(\varphi_0,q_0)\tilde{\varphi}\bigg]\,\dx\dt \\&\quad-\frac{\varepsilon}{8}\int_{0}^{T}\!\!\!\int_{\Tt}\left(\varphi- \tilde{\varphi}\right)^2\,\dx\dt-\frac{\varepsilon}{8}\int_{0}^{T}\!\!\!\int_{\Tt}\sum_{|j|=6}^{}\left( \partial^j\varphi-\partial^j\tilde{\varphi}\right)^2\,\dx\dt\\ &<\frac{1}{2}I_{(\varphi_0,q_0)}[\varphi]+\frac{1}{2}I_{(\varphi_0,q_0)}[\tilde{\varphi}]=I_0,
        \end{aligned}
        \end{equation}
        which is a contradiction.  Hence, $\varphi=\tilde{\varphi}$.
\end{proof}

The following proposition relates the solution of \eqref{vpp} with  problem \eqref{EL}.
\begin{pro}Fix  $(\varphi_0,q_0)\in \mathcal{H}^{+}_0\times H^{1}([0,T])$. Assume that $\varphi\in \mathcal{M}^{6}_\varepsilon$ solves \eqref{vpp}, and set $\Omega_\epsi=\{ (t,x)\in (0,T)\times\Tt: \varphi(t,x)>\varepsilon\}$. Then, $\varphi$ satisfies 
        \begin{equation}\label{prop1}
        \varepsilon\left( \varphi+\sum_{|j|=6}\partial^{2j}\varphi\right) +F_1(\varphi_0,q_0)=0\quad \text{in } \Omega_\epsi.
        \end{equation}
\end{pro}

\begin{proof}First, we prove that for all $v\in\mathcal{M}^{6}_\varepsilon$, we have 
        \begin{equation}\label{b-weak}
        \int_{0}^{T}\!\!\!\int_{\Tt}\bigg(\varepsilon \varphi(v-\varphi)+\epsi\sum_{|j|=6}^{}\partial^j\varphi\left( \partial^jv-\partial^j\varphi\right)  +F_1(\varphi_0,q_0)(v-\varphi )\bigg)\,\dx\dt\geq 0.
        \end{equation}
        
        Given  $\tau\in [0,1]$, we have  $\varphi+\tau(v-\varphi)=(1-\tau)\varphi+\tau v\in \mathcal{M}^{6}_\varepsilon$.  Thus, the real-valued function $i:[0,1]\to\Rr$ defined for \(\tau\in [0,1]\) by
        \begin{equation*}
        i(\tau)=I_{(\varphi_0,q_0)}[\varphi+\tau(v-\varphi)]
        \end{equation*}
        is well-defined and $C^1$. 

Because $\varphi$ solves \eqref{vpp}, \(i\) attains a minimum at \(\tau=0\); hence,  $i^\prime(0)\geq 0$. On the other hand, for $0<\tau\leq 1$,
we have        \begin{equation*}
        \begin{aligned}
        \frac{1}{\tau}(i(\tau)-i(0))&=\int_{0}^{T}\!\!\!\int_{\Tt} \bigg[ F_1(\varphi_0,q_0)(v-\varphi )+\varepsilon \bigg( \varphi(v-\varphi)+\sum_{|j|=6}^{}\partial^j\varphi\left( \partial^jv-\partial^j\varphi\right)\bigg)\bigg] \,\dx\dt\\&\quad+\varepsilon\int_{0}^{T}\!\!\!\int_{\Tt}\frac{\tau}{2}
\bigg( (v-\varphi)^2+\sum_{|j|=6}^{}\left( \partial^jv-\partial^j\varphi\right)^2\bigg)\,\dx\dt.
        \end{aligned}
        \end{equation*}
        Using the  estimate $i^\prime(0)\geq 0$ and letting $\tau\to0^+$ in the preceding identity, we obtain \eqref{b-weak}.
        
        Next, we fix $v_1\in C^\infty_c(\Omega)$ and observe that taking $s\in\Rr$ sufficiently close to $0$, we have $v=\varphi+sv_1\in\mathcal{M}^{6}_\varepsilon$. Then, by \eqref{b-weak}, we get
        \begin{equation*}
        s\int_{0}^{T}\!\!\!\int_{\Tt}\bigg[\varepsilon\bigg( \varphi v_1+\sum_{|j|=6}^{}\partial^j\varphi\,\partial^jv_1 \bigg) +F_1(\varphi_0,q_0)v_1 \bigg]\,\dx\dt\geq 0
        \end{equation*}
        for all such \(s\). Because the sign of $s$ is arbitrary, the preceding inequality implies \eqref{prop1}.
 \end{proof}

Next,  using the Lax--Milgram theorem, we prove the existence and uniqueness of solutions to
\begin{equation*}
\begin{cases}
\varepsilon( q-q^{\prime\prime})+F_2(\varphi_0,q_0)=0,\quad t\in[0,T],\\
q^\prime(0)=q^\prime(T)=0.
\end{cases}
\end{equation*} 
For $q_1$, $q_2$, $q\in H^{1}([0,T])$, we set
\begin{equation}\label{def-bi}
\begin{split}
B[q_1,q_2]=\int_{0}^{T}\varepsilon (q_1q_2+q_1^{\prime}q_2^{\prime})\,\dt
\enspace \text{ and } \enspace
\langle  f_{(\varphi_0,q_0)},q\rangle=\int_{0}^{T}F_2(\varphi_0,q_0)q\,\dt.
\end{split}
\end{equation}
We then consider the  problem of  finding $\bar q\in H^1([0,T])$ such that
\begin{equation}\label{bi}
B[\bar q,q]=\langle f_{(\varphi_0,q_0)},q\rangle \quad \text{for all }  q\in H^1([0,T]).
\end{equation}

\begin{pro} \label{prop:bipb}Fix \(\epsi>0\) and  $(\varphi_0,q_0)\in \mathcal{M}^{5}_{\varepsilon}\times H^1([0,T])$. Then, there exists a unique solution, $\bar q\in H^1([0,T])$, to \eqref{bi}. Moreover, \(\bar q\in  H^2([0,T])\) and there exists a positive constant, $C$, depending only on $\varphi_0$,  $q_0$, and \(\epsi\),  such that
        \begin{equation}\label{q-h1}
        \Vert\bar  q\Vert_{H^2([0,T])}\leq C.
        \end{equation}
\end{pro}

 \begin{proof} As we observed before, we have  $(\varphi_0,q_0)\in \mathcal{M}^{5}_{\varepsilon}\times H^1([0,T]) \subset C^2([0,T]\times\Tt)\times C([0,T])$. Moreover, because  $\varphi_x+1\geq \varepsilon$, we have $ F_2(\varphi_0,q_0)\in L^\infty([0,T])$. Hence,  $\langle  f_{(\varphi_0,q_0)},\cdot\rangle$ is a bounded linear functional on $H^1([0,T])$. On the other hand,  the Cauchy--Schwarz inequality and the definition of \(B\)  yield, for all  $q_1$, $q_2$, $q\in H_0^{1}([0,T])$,
that        \begin{equation}\label{lax}
        \begin{split}
        &|B[q_1,q_2]| \leq \epsi \Vert q_1\Vert_{H^1}\Vert q_2\Vert_{H^1},\\
        &\varepsilon \Vert q_1\Vert_{H^1}^2=\varepsilon (\Vert q_1\Vert_{L^2}^2+\Vert q^\prime_1\Vert_{L^2}^2)= B[q_1,q_1].
        \end{split}
        \end{equation} 
         Thus, by the Lax--Milgram theorem, we obtain the existence and uniqueness of solutions to \eqref{bi}.
        
We are left to prove \eqref{q-h1}. Let $\bar q\in H^1([0,T])$ be the solution of \eqref{bi}.  By Young's inequality, we have
        \begin{equation*}
        \begin{split}
        \varepsilon\big (\Vert\bar  q\Vert_{L^2}^2+\Vert\bar  q^\prime\Vert_{L^2}^2\big) &= B[\bar q,\bar q]=\langle  f_{(\varphi_0,q_0)},\bar q\rangle=\int_{0}^{T}F_2(\varphi_0,q_0)\bar q\,\dt\\ &\leq \frac{\varepsilon}{2}\Vert\bar
 q\Vert_{L^2}^2+\frac{T\Vert F_2(\varphi_0,q_0)\Vert_{L^\infty}^2 }{2\varepsilon}.
        \end{split}
        \end{equation*}
        The preceding inequality implies
         \begin{equation}\label{eq-q-h1}
        \Vert\bar  q\Vert_{H^1([0,T])}\leq C,
        \end{equation}
where the constant \(C\) depends only on  \(T\),   $\varphi_0$,  $q_0$, and \(\epsi\).

     Because $q\in H^1([0,T])$, $\bar{q}^{\prime\prime}$ is a distribution and, from \eqref{bi}, it follows that for any  $\phi \in C_0^\infty((0,T))$, we have
        \begin{equation*}
    \langle \bar{q}^{\prime\prime},\phi\rangle=-\int_{0}^{T}\bar{q}^\prime\phi^\prime\,\dt=-\int_{0}^{T}\left( \frac{F_2(\varphi_0,q_0)}{\varepsilon}-\bar{q}\right) \phi\,\dt. 
     \end{equation*}
     The preceding equality yields that $q^{\prime\prime}$ is a function and 
     \begin{equation}\label{second-der}
\bar{q}^{\prime\prime}=-\frac{F_2(\varphi_0,q_0)}{\varepsilon}+\bar{q}
\in L^\infty([0,T]).
     \end{equation}
     Thus,  $\bar{q}\in H^2([0,T])$   with  \eqref{q-h1}  by \eqref{eq-q-h1}  and  \eqref{second-der}. 
\end{proof}

\begin{remark}\label{rmk:onCs}
We observe that the dependence on \(\varphi_0\) and \(q_0\) of the constants in \eqref{m-bo}
and \eqref{q-h1} can be written as \( C_\epsi (\Vert \varphi_0\Vert_{H^{5}}
+ \Vert q_0\Vert_{H^1})\), where the constant \(C_\epsi\)
depends  on \(\epsi\) and \(T\) but not on \((\varphi_0,q_0)\).   
\end{remark}

Assume that Assumption~\ref{assumtionOnBoundsA} holds with
\(m_0\), \(m_T \in H^{5}(\Tt)\),  and let  \(0<\epsi\leq k_0\). In this setting, we  consider the operator $S:\mathcal{M}^{5}_\varepsilon\times H^{1}([0,T])\to \mathcal{M}^{5}_\varepsilon\times H^{1}([0,T])$ defined, for \((\varphi_0,q_0)\in\mathcal{M}^{5}_\varepsilon\times H^{1}([0,T]) \), by  
\begin{equation}\label{def_S}
S\begin{bmatrix}
\varphi _0 \\q_0
\end{bmatrix} =\begin{bmatrix}
\varphi ^*_0 \\q^*_0
\end{bmatrix},
\end{equation}
where $\varphi ^*_0$ and $q^*_0$  are  the unique solutions of \eqref{vpp} and \eqref{bi}, respectively.
\begin{definition}\label{def-wee} Let  \(m_0\), \(m_T \in H^{5}(\Tt)\)
satisfy Assumption~\ref{assumtionOnBoundsA}
  and let  \(0<\epsi\leq k_0\). A pair $(\varphi,q)\in \mathcal{M}^{6}_\varepsilon\times H^{1}([0,T])$ is   a weak solution to \eqref{eq: mot-2} if it is a  fixed point of the operator $S$ 
in \eqref{def_S}.
\end{definition}

\begin{pro}\label{con-comp} Assume that Assumption~\ref{assumtionOnBoundsA}
holds with
\(m_0\), \(m_T \in H^{5}(\Tt)\),  and let  \(0<\epsi\leq k_0\).  Then, the operator $S:\mathcal{M}^{5}_\varepsilon\times H^{1}([0,T])\to \mathcal{M}^{5}_\varepsilon\times H^{1}([0,T])$
        defined by \eqref{def_S} is continuous and compact. 
\end{pro}
\begin{proof}First, we prove the continuity of $S$. Let  $\varphi _0$, $\varphi _n\in \mathcal{M}^{5}_\varepsilon$ and $q_0$, $q_n\in H^{1}([0,T])$ be such that $\varphi _n\to\varphi _0$ in $H^{5}([0,T]\times\Tt)$ and $q_n\to q_0$ in $H^1([0,T])$. We aim at proving that 
 $\varphi^* _n\to\varphi ^*_0$ in $H^{5}([0,T]\times\Tt)$ and $q^*_n\to q^*_0$ in $H^{1}([0,T])$, where 
        \begin{equation*}
        \begin{bmatrix}
        \varphi ^*_0 \\q^*_0
        \end{bmatrix}=S\begin{bmatrix}
        \varphi _0 \\q_0
        \end{bmatrix} 
        \enspace \text{ and } \enspace\begin{bmatrix}
        \varphi ^*_n \\q^*_n
        \end{bmatrix} =S\begin{bmatrix}
        \varphi _n \\q_n
        \end{bmatrix}.
        \end{equation*}
By Propositions~\ref{prop:minpb}
and \ref{prop:bipb}, we then have 
\begin{equation*}
        \begin{aligned}
        &I_{(\varphi_0,q_0)}[\varphi_0^*]=\min\limits_{\varphi\in \mathcal{M}^{6}_\varepsilon} I_{(\varphi_0,q_0)}[\varphi],\quad I_{(\varphi_n,q_n)}[\varphi_n^*]=\min\limits_{\varphi\in \mathcal{M}^{6}_\varepsilon} I_{(\varphi_n,q_n)}[\varphi],
\end{aligned}
        \end{equation*}
and, for all \(q\in H^{1}([0,T]) \),

\begin{equation}\label{eq:solB}
        \begin{aligned} &B[q_0^*,q]=\langle f_{(\varphi_0,q_0)},q\rangle, \quad B[q_n^*,q]=\langle f_{(\varphi_n,q_n)},q\rangle.
        \end{aligned}
        \end{equation}
        
Using the definition of a minimizer,  it follows that        \begin{equation*}
        I_{(\varphi_0,q_0)}[ \varphi_0^*]+I_{(\varphi_n,q_n)}[ \varphi_n^*]  \leq I_{(\varphi_0,q_0)}\left[ \frac{\varphi_0^*+\varphi_n^*}{2}\right]+I_{(\varphi_n,q_n)}\left[ \frac{\varphi_0^*+\varphi_n^*}{2}\right].  
        \end{equation*}
        The preceding inequality combined  with Young's inequality gives
        \begin{equation}\label{ineq}
        \begin{aligned}
        &\frac{\varepsilon}{4} \int_{0}^{T}\!\!\!\int_{\Tt}\bigg(\left(\varphi_0^*-\varphi_n^*\right)^2 +\sum_{|j|=6}^{}\left( \partial^j\varphi_0^*- \partial^j\varphi_n^*\right)^2 \bigg) \,\dx\dt\\&\quad\leq\int_{0}^{T}\!\!\!\int_{\Tt}\frac{1}{2}\left(\varphi_0^*-\varphi_n^*\right)\big(F_1(\varphi_n,q_n)-F_1(\varphi_0,q_0)\big)\,\dx\dt\\ &\quad\leq \frac{\varepsilon}{8}\int_{0}^{T}\!\!\!\int_{\Tt}\left(\varphi_0^*-\varphi_n^*\right)^2\,\dx\dt+\frac{1}{2\varepsilon}
\int_{0}^{T}\!\!\!\int_{\Tt}(F_1(\varphi_n,q_n)-F_1(\varphi_0,q_0))^2 \bigg)\,\dx\dt.
        \end{aligned}
        \end{equation}

To prove that the last integral term  in \eqref{ineq}
converges to zero,   we first observe that (see  \eqref{def-F1F2})
\begin{equation}\label{eq:fifF1s}
\begin{aligned}
F_1(\varphi_n,q_n)-F_1(\varphi_0,q_0) =f_1 (\varphi_n,q_n,\varphi_0,q_0)
- \left( \frac{q_n}{ ((\varphi_n)_x+1)^{1-\alpha}}\right)_t 
+ \left( \frac{q_0}{ ((\varphi_0)_x+1)^{1-\alpha}}\right)_t,
\end{aligned}
\end{equation}
where
\begin{equation*}
\begin{aligned}
f_1 (\varphi_n,q_n,\varphi_0,q_0) & = - \left( \frac{(\varphi_n)_t}{ ((\varphi_n)_x+1)^{1-\alpha}}\right)_t
 +\frac{1}{2}\left(\frac{((\varphi_n)_t+q_n)^2}{ ((\varphi_n)_x+1)^{2-\alpha}}
\right)_x- (((\varphi_n)_x+1)^\mu)_x
\\
&\quad+ \left( \frac{(\varphi_0)_t}{ ((\varphi_0)_x+1)^{1-\alpha}}\right)_t
 -\frac{1}{2}\left(\frac{((\varphi_0)_t+q_0)^2}{ ((\varphi_0)_x+1)^{2-\alpha}}
\right)_x+ (((\varphi_0)_x+1)^\mu)_x.
\end{aligned}
\end{equation*}
Because \(\{\varphi_n\}_{n=1}^\infty\) and \(\{q_n\}_{n=1}^\infty\)
are bounded sequences  in $H^{5}([0,T]\times\Tt)$ and $H^1([0,T])$,
respectively, 
  we can find a positive constant,
$c$, such that
        \begin{equation}\label{boundedness}
        \sup\limits_{n\in\Nn}\big\{\Vert \varphi_0\Vert_{W^{2,\infty}}+\Vert
\varphi_n\Vert_{W^{2,\infty}}+\Vert q_0\Vert_{L^{\infty}}+\Vert
q_n\Vert_{L^{\infty}}\big\}\leq c
        \end{equation}        
by the embeddings  $\mathcal{M}^{5}_{\varepsilon}\subset
        C^{2}([0,T]\times\Tt)$ and $H^{1}([0,T])\subset
        C([0,T])$. From \eqref{boundedness} and the
uniform lower bounds \((\varphi_0)_x+1 \geq \epsi\)  and  \((\varphi_n)_x+1 \geq \epsi\), we conclude that there exists a positive constant,
\(\tilde c\), independent of \(n\in\Nn\), such that     
\begin{equation*}
\begin{aligned}
\sup_{n\in\Nn} \Vert f_1 (\varphi_n,q_n,\varphi_0,q_0)\Vert_{L^\infty}
\leq \tilde c.
\end{aligned}
\end{equation*}
Moreover, as  $\varphi _n\to\varphi _0$
in $C^2([0,T]\times\Tt)$ and $q_n\to q_0$ in $C([0,T])$, we also
have
\begin{equation*}
\begin{aligned}
\lim_{n\to\infty} f_1 (\varphi_n,q_n,\varphi_0,q_0) =0 \enspace  \text{ pointwise
in
 } [0,T]\times\Tt.
\end{aligned}
\end{equation*}
Thus, by the Lebesgue dominated convergence theorem, it follows that
\begin{equation}
\label{eq:f1to0}
\begin{aligned}
\lim_{n\to\infty}\int_{0}^{T}\!\!\!\int_{\Tt}\big| f_1 (\varphi_n,q_n,\varphi_0,q_0)\big|^2\,\dx\dt
= 0.
\end{aligned}
\end{equation}

Similarly, 
 \eqref{boundedness} and the
uniform lower bounds \((\varphi_0)_x+1 \geq \epsi\)  and  \((\varphi_n)_x+1
\geq \epsi\) yield, for some positive constant, \(\tilde C\),
independent of \(n\), that%
\begin{equation}
\label{eq:limF1}
\begin{aligned}
&\limsup_{n\to\infty}\int_{0}^{T}\!\!\!\int_{\Tt}\left| \left( \frac{q_n}{ ((\varphi_n)_x+1)^{1-\alpha}}\right)_t 
- \left( \frac{q_0}{ ((\varphi_0)_x+1)^{1-\alpha}}\right)_t\right|^2\,\dx\dt
\\
&\quad\leq \tilde C\limsup_{n\to\infty}\bigg(\int_{0}^{T}\!\!\!\int_{\Tt}
\big(|q_n -q_0|^2 + |(q_n)_t - (q_0)_t|^2\big) \,\dx\dt\\&\hskip22mm+\int_{0}^{T}\!\!\!\int_{\Tt}
|(q_0)_t|^2
 \big|((\varphi_n)_x+1)^{-1+\alpha}
- ((\varphi_0)_x+1)^{-1+\alpha} \big|^2 \,\dx\dt\\&\hskip22mm+\int_{0}^{T}\!\!\!\int_{\Tt} |q_0|^2
 \big|\big(((\varphi_n)_x+1)^{-1+\alpha}\big)_t
- \big(((\varphi_0)_x+1)^{-1+\alpha}\big)_t \big|^2\,\dx\dt\bigg)\\
&\quad=0,
\end{aligned}
\end{equation}
where we also used the Lebesgue dominated convergence theorem
together with the convergences   $\varphi _n\to\varphi _0$
in $C^2([0,T]\times\Tt)$ and $q_n\to q_0$ in $H^1([0,T])$. 

From \eqref{ineq}, \eqref{eq:fifF1s}, \eqref{eq:f1to0}, and \eqref{eq:limF1},
we deduce that
        \begin{equation*}
       \lim_{n\to\infty} \int_{0}^{T}\!\!\!\int_{\Tt}\bigg(\left(\varphi_0^*-\varphi_n^*\right)^2 +\sum_{|j|=6}^{}\left( \partial^j\varphi_0^*- \partial^j\varphi_n^*\right)^2 \bigg) \,\dx\dt= 0.
        \end{equation*} 
        Consequently, by the Gagliardo--Nirenberg interpolation inequality, we have   $\varphi_0^*\to\varphi_n^*$ in $H^{6}([0,T]\times \Tt)$.
        
On the other hand, recalling \eqref{def-bi} and \eqref{def-F1F2}, we conclude from \eqref{eq:solB} and Young's inequity that        \begin{equation*}
        \begin{aligned}
        \varepsilon\big(\Vert q_n^*-q_0^*\Vert_{L^2}^2+\Vert (q_n^*)^\prime-(q_0^*)^\prime\Vert_{L^2}^2\big)&=B[q_n^*-q_0^*,q_n^*-q_0^*]\\
        & =\int_{0}^{T}\!\!\!\int_{\Tt}({F}_2(\varphi_n,q_n)-{F}_2(\varphi_0,q_0))( q_n^*-q_0^*)\,\dx\dt\\
&\leq \frac{\varepsilon}{2} \Vert q_n^*-q_0^*\Vert_{L^2}^2+\frac{1}{2\varepsilon}\int_{0}^{T}\!\!\!\int_{\Tt}|f_2(\varphi_n,q_n,\varphi_0,q_0)|^2\,\dx\dt,
        \end{aligned}
        \end{equation*}
where, arguing as in \eqref{eq:f1to0},  
\begin{equation*}
\begin{aligned}
&\lim_{n\to\infty}\int_{0}^{T}\!\!\!\int_{\Tt}\big| f_2 (\varphi_n,q_n,\varphi_0,q_0)\big|^2\,\dx\dt
\\ &\quad= \lim_{n\to\infty}\int_{0}^{T}\!\!\!\int_{\Tt} \bigg|
  \frac{(\varphi_n)_t - q_n}{
((\varphi_n)_x+1)^{1-\alpha}} - \frac{(\varphi_0)_t - q_0}{
((\varphi_0)_x+1)^{1-\alpha}} \bigg|^2 \,\dx\dt
= 0.
\end{aligned}
\end{equation*}
Hence,   $q^\ast_n\to q^\ast_0$ in $H^1([0,T])$, which conclude the proof of the continuity of \(S\).        
        
The compactness of \(S\) follows from the estimates in \eqref{m-bo}
and \eqref{q-h1} (also see Remark~\ref{rmk:onCs})\ together with the Rellich--Kondrachov theorem.
\end{proof}

Next, we prove the existence and uniqueness of solutions to \eqref{eq: mot-2} in the sense of Definition~\ref{def-wee}.
The existence is hinged on the  following  Schaefer's fixed-point theorem (a proof of this result can be found,
for instance, in
 \cite[Theorem~6.2]{FGT1}).

\begin{theorem}\label{Sch4-1}
Let $\mathcal{C}$ be a convex and closed subset of a Banach space
such that  $0\in \mathcal{C}$. Assume that $\mathcal{S}\colon \mathcal{C}
\to \mathcal{C}$ is continuous and compact,  for which 
the set
\begin{equation*}
\Big\{
w\in \mathcal{C}\ |\ w=\lambda\, \mathcal{S}[w]\ \ \mbox{for some }\lambda
\in [0,1]
\Big\}
\end{equation*}
is bounded. Then, \(\mathcal{S}\) has a fixed point in \(\mathcal{C}\);
that is, there exists $w\in \mathcal{C}$
such that $w=\mathcal{S}[w]$.
\end{theorem}

\begin{teo}\label{weak-ex-teo} Let $0<\varepsilon\leq k_0$.   Suppose that \(\alpha\in (0,2)\) satisfies  $\alpha<\mu+1$ and that Assumption~\ref{assumtionOnBoundsA} holds with  \(m_0\), \(m_T \in H^{5}(\Tt)\). Then, \eqref{eq: mot-2} has a unique weak solution in the sense of Definition~\ref{def-wee}.
\end{teo}
\begin{proof} We first deal with the existence of weak solutions. Set $\tilde{\mathcal{M}}=\mathcal{M}^{5}_\varepsilon-\varphi^0$, where $\varphi^0$ is defined by \eqref{phiTest}, and consider the operator $\tilde{S}:\tilde{\mathcal{M}}\times H^1([0,T])\to \tilde{\mathcal{M}}\times H^1([0,T])$ given by
        \begin{equation}\label{S-bar}
        \tilde{S}\begin{bmatrix}
        \varphi\\
        q
        \end{bmatrix}=S\begin{bmatrix}
        \varphi+\varphi^0\\
        q
        \end{bmatrix}-\begin{bmatrix}
        \varphi^0\\
        0
        \end{bmatrix},  \quad (\varphi,q) \in \tilde{\mathcal{M}}\times H^1([0,T]),
        \end{equation}
        where, we recall, \(S\) is the operator introduced in
\eqref{def_S}. Note that proving that \eqref{eq: mot-2}
has a  weak solution in the sense of Definition~\ref{def-wee} is equivalent to proving that \(\tilde S\) has a fixed point. 

We have  \((0,0)\in
\tilde{\mathcal{M}}\times H^1([0,T])\) and, using Proposition~\ref{con-comp},
\(\tilde S\) is a continuous and compact operator on \(\tilde{\mathcal{M}}\times H^1([0,T])\). Moreover, let  $\lambda\in[0,1]$ and assume that $(\bar{\varphi}_\lambda,\bar{q}_\lambda)\in \tilde{\mathcal{M}}\times H^1([0,T])$ satisfies the identity 
        \begin{equation}\label{lin-oper}
        \begin{bmatrix}
        \bar{\varphi}_\lambda\\
        \bar{q_\lambda}
        \end{bmatrix}=\lambda \bar{S} \begin{bmatrix}
        \bar{\varphi}_\lambda\\
        \bar{q}_\lambda
        \end{bmatrix}.
        \end{equation}
        Note that by definitions of \(S\),  \(\tilde S\), and
        \(\varphi^0\), we
have \(\bar \varphi_\lambda\in H^{6}([0,T]\times \Tt)\) and
\(\bar q_\lambda\in H^2([0,T])\) by Propositions~\ref{prop:minpb} and \ref{prop:bipb}.

 If $\lambda=0$, then $(\bar{\varphi}_\lambda,\bar{q}_\lambda)=(0,0)$. If $0 <\lambda\leq 1$, we set
        \begin{equation}\label{def-phi}
        \varphi_\lambda=\bar{\varphi}_\lambda+\varphi^0\in \mathcal{M}^{6}_\varepsilon
\enspace \text{ and } \enspace q_\lambda = \bar q_\lambda \in
H^1([0,T])
        \end{equation}
        and note that \eqref{S-bar}, \eqref{lin-oper}, and \eqref{def-phi} yield the identity       \begin{equation*}
        S\begin{bmatrix}
        \varphi_\lambda\\
        q_\lambda
        \end{bmatrix}=\frac{1}{\lambda}\begin{bmatrix}
        \varphi_\lambda+(\lambda-1)\varphi^0\\
        q_\lambda
        \end{bmatrix}.
        \end{equation*}
        Using the definition of  $S$ and   setting $\tilde{\varphi}_\lambda=\frac{1}{\lambda}(\varphi_\lambda+(\lambda-1)\varphi^0)$ and  $\tilde{q}_\lambda=\frac{1}{\lambda}q_\lambda$,
we have \(\tilde{\varphi}_\lambda \in\mathcal{M}^{6}_\varepsilon\) and \(\tilde{q}_\lambda\in H^1([0,T])\) and we may use   \eqref{b-weak} and \eqref{bi}, with \((\varphi_0,q_0) = ({\varphi}_\lambda,q_\lambda)\),
 \(\varphi = \tilde \varphi_\lambda\), and \(\bar q = \tilde
 q_\lambda\),
 to conclude that        
        \begin{align}
        &\int_{0}^{T}\!\!\!\int_{\Tt}\bigg(\varepsilon \tilde{\varphi}_\lambda(v-\tilde{\varphi}_\lambda)+\epsi
\sum_{|j|=6}^{}\partial^j\tilde{\varphi}_\lambda\left( \partial^j v-\partial^j\tilde{\varphi}_\lambda\right)  +F_1(\varphi_\lambda,q_\lambda)(v-\tilde{\varphi}_\lambda )\bigg)\,\dx\dt\geq 0,\label{main-exist1}\\
        &\int_{0}^{T}\Big(\varepsilon (\tilde{q}_\lambda q+\tilde{q}_\lambda^{\prime}q^{\prime})
        +F_2(\varphi_\lambda,q_\lambda)q\Big) \,\dt=0,\label{main-exist2}
        \end{align}
        for all $v\in \mathcal{M}^{6}_\varepsilon$ and  $q\in H^1([0,T])$.
        Next, we take $v=\varphi_\lambda$  in \eqref{main-exist1} multiplied by \(\frac{\lambda^2}{\lambda-1}\) and $q=q_\lambda$ in \eqref{main-exist2}
 multiplied by $\lambda$, and observe  that
\(\varphi_\lambda - \tilde\varphi_\lambda = \tfrac{\lambda-1}{\lambda}(\varphi_\lambda
- \varphi^0)\) and \(\lambda-1\leq 0\). Then, adding the two resulting inequalities 
          and recalling the definition
of the operator  \(A\) in \eqref{def_A},
we obtain 
\begin{equation*}
\begin{aligned}
&\int_{0}^{T}\!\!\!\int_{\Tt}  \varepsilon\bigg( \varphi_\lambda^2+\sum_{|j|=6}^{}(\partial^j\varphi_\lambda)^2
+ q_\lambda^2+(q_\lambda^\prime)^2\bigg)\,\dx\dt+\left( A\begin{bmatrix}
        \varphi_\lambda \\q_\lambda     
        \end{bmatrix}, \begin{bmatrix}
        \varphi_\lambda - \varphi^0 \\q_\lambda
        \end{bmatrix}\right)_{L^2\times L^2}\\
        &\quad \leq \epsi\int_{0}^{T}\!\!\!\int_{\Tt} \bigg[
        (\lambda-1)\bigg((\varphi^0)^2+\sum_{|j|=6}^{}(\partial^j\varphi^0)^2
\bigg)-(\lambda-2)\bigg(\varphi_\lambda\varphi^0
+ \sum_{|j|=6}^{}\partial^j\varphi_\lambda \partial^j\varphi^0\bigg)
\bigg]\,\dx\dt.
 \end{aligned}
\end{equation*}

Arguing as in the proof of Proposition~\ref{apri-1}, we conclude
from the preceding estimate that         
\begin{equation} \label{key-for apriori} 
\begin{aligned}
     & \int_{0}^{T}\!\!\!\int_{\Tt}  \varepsilon\bigg(  \varphi_\lambda^2+\sum_{|j|=6}^{}(\partial^j\varphi_\lambda)^2
+ q_\lambda^2+(q_\lambda^\prime)^2\bigg)\,\dx\dt\\
&\quad + \int_{0}^{T}\!\!\!\int_{\Tt} \frac{((\varphi_\lambda)_t+q_\lambda)^2}{ ((\varphi_\lambda)_x+1)^{1-\alpha}}+\frac{((\varphi_\lambda)_t+q)^2}{ ((\varphi_\lambda)_x+1)^{2-\alpha}}+((\varphi_\lambda)_x+1)^{\mu+1}\,\dx\dt\leq C,
\end{aligned}
        \end{equation}
        where $C$ does not depend on $\lambda$. The  estimate
in \eqref{key-for apriori} and  the Gagliardo--Nirenberg
interpolation         inequality imply that  $(\varphi_{\lambda},q_\lambda)$ is  bounded in $\mathcal{M}^{6}_\varepsilon\times H^1([0,T])$
uniformly in \(\lambda\).  Therefore, recalling \eqref{def-phi}, we may apply Theorem~\ref{Sch4-1} to conclude    that $\bar{S}$ has a fixed point. Thus,    \eqref{eq: mot-2} has a weak solution
in the sense of Definition~\ref{def-wee}.

Finally, we prove the uniqueness of such weak solutions. Suppose that the operator $S$ has two fixed points, $(\varphi_1,q_1)$ and $(\varphi_2,q_2)$. Then, by Propositions~\ref{prop:minpb} and \ref{prop:bipb},
we have $(\varphi_1,q_1)$, $(\varphi_2,q_2)
\in \mathcal{M}^{6}_\varepsilon \times  H^2([0,T])$.
Moreover, setting
\begin{equation*}
\begin{aligned}
B&= \varepsilon\int_{0}^{T}\!\!\!\int_{\Tt}
\bigg(( \varphi_1- \varphi_2)^2+\sum_{|j|=6}^{} (\partial^j\varphi_1-\partial^j\varphi_2)^2\bigg)
\,\dx\dt\\&\quad+\varepsilon\int_{0}^{T} \bigg( (q_1-q_2)^2+(q_1^{\prime}-q_2^{\prime})^2\bigg)
\,\dt +\left( A\begin{bmatrix}
        \varphi_1 \\q_1
        \end{bmatrix}-A\begin{bmatrix}
        \varphi_2 \\q_2
        \end{bmatrix}, \begin{bmatrix}
        \varphi_1 \\q_1
        \end{bmatrix}-\begin{bmatrix}
        \varphi_2 \\q_2
        \end{bmatrix}\right)_{L^2\times L^2} ,
\end{aligned}
\end{equation*}
we have \(B\geq 0\) by Proposition~\ref{mon-A}.
On the other hand, we may use    \eqref{b-weak}
and \eqref{bi} twice, where %
\begin{enumerate}
\item[(i)]  \((\varphi_0,q_0) = ({\varphi}_1,q_1)\), \(\varphi=\varphi_1\),
\(v=\varphi_2\), \(\bar q=q_1\), and \(q=q_2\);

\item[(ii)]  \((\varphi_0,q_0) = ({\varphi}_2,q_2)\), \(\varphi=\varphi_2\),
\(v=\varphi_1\), \(\bar q=q_2\), and \(q=q_1\). 
\end{enumerate}
Adding the two  inequalities resulting from (i) and (ii), we
get \(-B\geq 0\). Thus, \(B\equiv 0\) that, together with Proposition~\ref{mon-A},  yields $(\varphi_1,q_1)=(\varphi_2,q_2)$.
\end{proof}

 The next  propositions provide  a priori estimates for the weak and classical solutions to \eqref{eq: mot-2}. 
\begin{proposition}\label{p2} Let $\kappa>1$
and   $0<\varepsilon\leq
k_0$.   Suppose that $\alpha< \mu+1$ and that Assumption~\ref{assumtionOnBoundsA} holds with  \(m_0,m_T \in H^{5}(\Tt)\).  If $(\varphi,q)\in  \mathcal{M}^{6}_\varepsilon\times H^1([0,T])$ solves \eqref{eq: mot-2} in the sense  of Definition~\ref{def-wee}, then there exists a positive constant, $C$, such that
        \begin{equation}\label{p2esti: 1}
        \begin{split}
        &\int_{0}^{T}\!\!\!\int_{\mathbb{T}} \varepsilon\bigg(  \varphi^2+\sum_{|j|=6}^{}(\partial^j\varphi)^2+q^2+(q^\prime)^2\bigg) \,\dx\dt\leq C,\\ 
        &       \int_{0}^{T}\!\!\!\int_{\mathbb{T}} ( \varphi_x+1)^{\mu+1}\,\dx\dt\leq C.
        \end{split}
        \end{equation}
        Moreover,
for \(\alpha\in(0,2)\), we have
\begin{align}
&\int_{0}^{T}\!\!\!\int_{\mathbb{T}}\Big(|\varphi_t|^\frac{2(\mu+1)}{\mu+2-\alpha}+|q|^\frac{2(\mu+1)}{\mu+2-\alpha}
\Big)\,\dx\dt\leq C \qquad\text{if } \alpha\in(0,1],\label{p2esti: 2}\\
& \int_{0}^{T}\!\!\!\int_{\mathbb{T}}\Big(|\varphi_t|^\frac{2(\mu+1)}{\mu+3-\alpha}+|q|^\frac{2(\mu+1)}{\mu+3-\alpha}\Big)\,\dx\dt\leq
C \qquad\text{if } \alpha\in[1,2).\label{p2esti: 3}
\end{align}
\end{proposition}

\begin{proof} Combining the arguments used to establish  \eqref{key-for apriori} and  Proposition \ref{apri-1}, we deduce the result.
\end{proof}
\begin{cor}\label{lim} Under the assumptions of Theorem~\ref{weak-ex-teo},
let $(\varphi_\varepsilon,q_\varepsilon)\in \mathcal{M}^{6}_\varepsilon\times H^1([0,T])$ be the weak solution of \eqref{eq: mot-2} in the sense of Definition \ref{def-wee}. Set
\begin{equation*}
\begin{aligned}
\kappa=\begin{cases}
\min\big\{\frac{2(\mu+1)}{\mu+2-\alpha},\mu+1\big\} & \text{if } \alpha\in(0,1],\\
\frac{2(\mu+1)}{\mu+3-\alpha} & \text{if } \alpha\in(1,2),
\end{cases}
\end{aligned}
\end{equation*}
and recall the set \(\mathcal{A}_0^\kappa\) introduced in \eqref{def-A_0}.
Then,   
there exist a subsequence, \(\{(\varphi_{\varepsilon_n},q_{\varepsilon_n})\}_{n=1}^\infty\), and a pair  $(\varphi,q)\in \mathcal{A}_0^{\kappa}\times L^\kappa([0,T])$ such that $(\varphi_{\varepsilon_n},q_{\varepsilon_n})\rightharpoonup(\varphi,q)$ in $W^{1,\kappa}([0,T]\times\Tt)\times L^\kappa([0,T])$.
\end{cor}
\begin{proof} We first note that \(\kappa>1\), $\mathcal{M}^{6}_\varepsilon\subset\mathcal{A}_0^{\kappa}$, and  $\mathcal{A}_0^{\kappa}$ is a closed and convex subset of $W^{1,\kappa}([0,T]\times\Tt)$. Thus, the conclusion follows from the Poincar\'e--Wirtinger inequality together with the  a priori estimates in \eqref{p2esti: 1}--\eqref{p2esti: 3}.   
\end{proof}
\begin{pro}\label{p2-cl} Assume that \(\alpha\in (0,2)\) satisfies  $\alpha<\mu+1$ and  that Assumption~\ref{assumtionOnBoundsA} holds. If  $(\varphi,q)\in  C^{12}([0,T])\times (C^2([0,T])\cap H^1_0([0,T]))$ solves \eqref{eq: mot-2}--\eqref{eq: mot-2-bound} in the classical sense, then there exists a positive constant, $C$, for which the estimates \eqref{p2esti: 1}, \eqref{p2esti: 2}, and \eqref{p2esti: 3} hold.
\end{pro}
\begin{proof} We start by adding   the first equation  in \eqref{eq: mot-2}  multiplied by  $(\varphi-\varphi^0)$  to the second equation multiplied by  $q$. Then, we integrate the resulting
identity over \([0,T]\times
\Tt\).  Using the boundary conditions \eqref{eq: mot-2-bound}, an integration by parts,  and  arguments similar to those of the proof of Proposition \ref{apri-1}, we conclude the proof.
\end{proof}

We are now able to  prove Theorem \ref{con-ex}.
\begin{proof} [Proof of Theorem \ref{con-ex}]

For each  $0<\varepsilon\leq
k_0$,  let $(\varphi_\varepsilon,q_\varepsilon)\in \mathcal{M}^{6}_\varepsilon\times
H^1([0,T])$ be the weak solution of \eqref{eq: mot-2} in the
sense of Definition \ref{def-wee}, given by Theorem~\ref{weak-ex-teo}.
Moreover,  let \(\{(\varphi_{\varepsilon_n},q_{\varepsilon_n})\}_{n=1}^\infty\)
and   $(\varphi,q)\in \mathcal{A}_0^{\kappa}\times L^\kappa([0,T])$
be  given by Proposition~\ref{p2-cl}, so that $(\varphi_{\varepsilon_n},q_{\varepsilon_n})\rightharpoonup(\varphi,q)$
in $W^{1,\kappa}([0,T]\times\Tt)\times L^\kappa([0,T])$.

Fix
 $(\psi,\varpi)\in D(A)$ with \(\psi\in H^6([0,T]\times\Tt)\). \ By definition of \(D(A)\),  there exists \(\delta>0\) such that \(\psi_x+1\geq
\delta\). Then, because \(\lim_{n\to\infty}\epsi_n=0\), we have
\(\psi\in \mathcal{M}^{6}_{\varepsilon_n}\) for all \(n\in\Nn\)
sufficiently large; without loss of generality, we may assume
that this inclusion holds for all \(n\in\Nn\).
 
Using monotonicity of \(A\) proved in Proposition~\ref{mon-A}, we have 
\begin{equation}\label{eq:monAepsi}
\begin{aligned}
0&\leq \varepsilon_n\int_{0}^{T}\!\!\!\int_{\Tt}
\bigg(( \psi- \varphi_{\epsi_n})^2+\sum_{|j|=6}^{} (\partial^j\psi-\partial^j\varphi_{\epsi_n})^2\bigg)
\,\dx\dt\\&\quad+\varepsilon_n\int_{0}^{T} \big( (\varpi-q_{\epsi_n})^2+(\varpi^{\prime}-q_{\epsi_n}^{\prime})^2\big)
\,\dt +\left( A\begin{bmatrix}
        \psi \\\varpi
        \end{bmatrix}-A\begin{bmatrix}
        \varphi_{\epsi_n} \\q_{\epsi_n}
        \end{bmatrix}, \begin{bmatrix}
        \psi \\\varpi
        \end{bmatrix}-\begin{bmatrix}
        \varphi_{\epsi_n} \\q_{\epsi_n}
        \end{bmatrix}\right)_{L^2\times L^2} \\
        &= a_{\epsi_n}^1 - a_{\epsi_n}^2 -a_{\epsi_n}^3 +\left( A\begin{bmatrix}
        \psi \\\varpi
        \end{bmatrix}, \begin{bmatrix}
        \psi \\\varpi
        \end{bmatrix}-\begin{bmatrix}
        \varphi_{\epsi_n} \\q_{\epsi_n}
        \end{bmatrix}\right)_{L^2\times L^2} ,
\end{aligned}
\end{equation}
where 
        \begin{equation*}
        \begin{aligned}
        a_{\epsi_n}^1&=\varepsilon_n \int_{0}^{T}\!\!\!\int_{\Tt}\bigg(
\psi(\psi-\varphi_{\epsi_n})+\sum_{|j|=6}^{}\partial^j\psi\big(
\partial^j\psi-\partial^j \varphi_{\epsi_n}\big) \bigg)\,\dx\dt\\&\quad+\varepsilon_n
\int_{0}^{T}\Big(
\varpi (\varpi-q_{\epsi_n})+\varpi^{\prime}(\varpi^{\prime}-
q_{\epsi_n}^{\prime})
        \Big)\,\dt, \\
  a_{\epsi_n}^2&= \epsi_n \int_{0}^{T}\!\!\!\int_{\Tt} \bigg(
\varphi_{\epsi_n}(\psi-\varphi_{\epsi_n})+\sum_{|j|=6}^{}\partial^j\varphi_{\epsi_n}\big(
\partial^j\psi-\partial^j\varphi_{\epsi_n}\big) \bigg)  \,\dx\dt
\\&\quad +\int_{0}^{T}\!\!\!\int_{\Tt}F_1(\varphi_{\epsi_n},q_{\epsi_n})(\psi-\varphi_{\epsi_n})\,\dx\dt,    \\
a_{\epsi_n}^3&=\epsi_n\int_{0}^{T}\big( (q_{\varepsilon_n} (\varpi-q_{\varepsilon_n})+q_{\varepsilon_n}^{\prime}(\varpi^{\prime}-q_{\varepsilon_n}^{\prime})\big)
        \, \dt+\int_{0}^{T}F_2(\varphi_{\varepsilon_n},q_{\varepsilon_n})(\varpi-q_{\varepsilon_n})\,\dt.   
        \end{aligned}
        \end{equation*}
By     \eqref{b-weak}
and \eqref{bi}, we have \(a_{\epsi_n}^2 \geq 0\) and \(a_{\epsi_n}^3=0\).  Thus, letting \(n\to\infty\) in \eqref{eq:monAepsi},  using the convergence
 $(\varphi_{\varepsilon_n},q_{\varepsilon_n})\rightharpoonup(\varphi,q)$
in $W^{1,\kappa}([0,T]\times\Tt)\times L^\kappa([0,T])$, and
observing that \(A
        [\psi,\varpi]\in C([0,T]\times\Tt)\times C([0,T])\) because   $(\psi,\varpi)\in D(A)$,           we conclude
that %
\begin{equation*}
\begin{aligned}
0&\leq \lim_{n\to\infty} \bigg[ a_{\epsi_n}^1 +\left(
A\begin{bmatrix}
        \psi \\\varpi
        \end{bmatrix}, \begin{bmatrix}
        \psi \\\varpi
        \end{bmatrix}-\begin{bmatrix}
        \varphi_{\epsi_n} \\q_{\epsi_n}
        \end{bmatrix}\right)_{L^2\times L^2} \bigg] \\&= 
      \int_{0}^{T}\!\!\!\int_{\Tt} F_1(\psi,\varpi)(\psi-\varphi)\,\dx\dt
+\int_{0}^{T}
F_2(\psi,\varpi)(\varpi-q)\dt\\ &= \left\langle A\begin{bmatrix}
        \psi \\ \varpi
        \end{bmatrix}
        , \begin{bmatrix}
        \psi \\ \varpi
        \end{bmatrix}-\begin{bmatrix}
        \varphi \\q
        \end{bmatrix}\right\rangle_{L^{\kappa'}\times L^{\kappa'},
L^\kappa \times L^k}. 
\end{aligned}
\end{equation*}
Using a density argument together with the Lebesgue dominated
convergence theorem, we conclude that the preceding estimate holds
for all 
 $(\psi,\varpi)\in D(A)$, which proves that \((\varphi,q)\) is a weak solution to \eqref{opA: 0}
in the sense of Definition~\ref{def-weak}.
\end{proof} 

Next, we prove Proposition \ref{prorelc}; that is,  if the weak solution (in the sense of Definition~\ref{def-weak}) to \eqref{opA: 0} is smooth enough, then we can construct a classical solution to \eqref{planningCS}.
\begin{proof}[Proof of Proposition \ref{prorelc}]
        
Assume that $(\varphi,q)\in C^{2}([0,T]\times \Tt)\times C^{1}([0,T])$
is a weak
solution to \eqref{opA: 0}
in the sense of Definition~\ref{def-weak} such that \(\varphi_x+1
>0\). In particular, $(\varphi,q)\in D(A)$ and there exists \(\delta>0\) such
\begin{equation}\label{eq:ffi+}
\begin{aligned}
\varphi_x +\ 1 >\delta.
\end{aligned}
\end{equation}
Moreover, for
all $(\psi,\varpi)\in D(A)$, we   have
\eqref{def--weak-sol}.    
Using Morrey's embedding  theorem
as before, it follows that 
 \(A[\psi,\varpi]\in C([0,T]\times\Tt)\times C([0,T])\) and
\eqref{def--weak-sol} can be rewritten as
\begin{equation}
\label{eq:regws}
\begin{aligned} 
      \int_{0}^{T}\!\!\!\int_{\Tt} F_1(\psi,\varpi)(\psi-\varphi)\,\dx\dt
+\int_{0}^{T}
F_2(\psi,\varpi)(\varpi-q)\,\dt\geq 0.
\end{aligned}
\end{equation}

Fix \(\tilde \psi \in C^\infty_c((0,T)\times\Tt))\), with \(\int_\Tt
\tilde\psi\, \dx = 0\), and \(\tilde q\in C^{1}_0((0,T)) \). Then, using \eqref{eq:ffi+}, we can use \eqref{eq:regws} with
\begin{equation*}
\begin{aligned}
\psi = \tau\tilde\psi +\varphi \quad \text{ and } \quad \varpi
= \tau \tilde q + q
\end{aligned}
\end{equation*}
for all \(\tau>0\) with \(|\tau|\)
sufficiently small.
Hence,%
\begin{equation}
\label{eq:regws2}
\begin{aligned} 
      \int_{0}^{T}\!\!\!\int_{\Tt} F_1( \tau\tilde\psi +\varphi,\tau \tilde q + q)\tilde
\psi\,\dx\dt
+\int_{0}^{T}
F_2( \tau\tilde\psi +\varphi,\varpi)\tilde q\,\dt\geq 0
\end{aligned}
\end{equation}
for all such \(\tau\). Letting \(\tau\to0\) in \eqref{eq:regws2} and using the Lebesgue
dominated convergence theorem, the arbitrariness of \((\tilde\psi,
\tilde q)\) yields  $F_1
(\varphi, q)=0$ and   $F_2(\varphi, q)=0$; that is, $(\varphi,q)$ is a classical solution to \eqref{opA: 0}. Therefore, by Proposition~\ref{prorelc1}, we conclude the proof.
\end{proof}

        \section{Some Explicit  Examples  Of  The Hughes'  Model}
\label{sec4}
In this section, we present an application of our approach to the one-dimensional  Hughes' model. This model describes the evolution of a population with density $\rho$. This evolution is determined by the  system
\begin{equation}\label{Hu1}
\begin{cases} 
-\rho_t+ (\rho f^2(\rho)\Psi_x)_x =0    \\ f(\rho)|\Psi_x|=1 \\\rho(0,x)=\rho_0(x).
\end{cases}
\end{equation}
 Using a potential to interpret the first equation in \eqref{Hu1}, we deduce that there exists a function, $\varphi$, such that
\begin{equation}\label{rho}
\begin{cases}
\rho=\varphi_x\\
\rho f^2(\rho)\Psi_x=\varphi_t.
\end{cases}
\end{equation} 
Because $x$ is scalar, the second equation in \eqref{Hu1} is equivalent to
\begin{equation*}
 (f(\rho)\Psi_x)^2=1.
\end{equation*}
Then,  using \eqref{rho} and taking into consideration the initial condition for $\rho$, we get
%{\color{red}\begin{equation*}
%       (\rho f^2(\rho)\Psi_x)^2=\rho^2 f^2(\rho)
%       \end{equation*}}
\begin{equation}\label{newHu}
\begin{cases}
\varphi_t^2=\varphi_x^2f^2(\varphi_x)\\
\varphi(0,x)=\int_{-\infty}^{x}\rho_0(\tau)\,\d\tau.
\end{cases}
\end{equation}
Next,  we explicitly solve \eqref{newHu} for a  particular $f$.
\begin{example}\label{ex} Suppose that $\rho_0\in C(\Rr)$, which implies that $\varphi(0,\cdot)$ is  Lipschitz continuous in $\Rr$. The original Hughes' model is given by \eqref{Hu1} with $f(p)=1-p$. 
First, we consider the increasing quaintly case; that is, instead of \eqref{newHu}, we have
\begin{equation}\label{newHuex1}
\begin{cases}
\varphi_t-\varphi_xf(\varphi_x)=0\\
\varphi(0,x)=\int_{-\infty}^{x}\rho_0(\tau)\,\d\tau.
\end{cases}
\end{equation}
Notice that \eqref{newHuex1} is a Hamilton--Jacobi equation with the Hamiltonian $H(p)=-pf(p)=p^2-p$.
Because $H$ is convex and $\varphi(0,\cdot)$ is  Lipschitz, we
may use the Hopf-Lax formula to  solve \eqref{newHuex1}
explicitly:\begin{equation}\label{sol1}
\varphi(t,x)=\min\limits_{y}\left\lbrace tL\left( \frac{x-y}{t}\right) + \int_{-\infty}^{y}\rho_0(\tau)\,\d\tau\right\rbrace,
\end{equation}
where $L$ is Legendre transform of $H$. 

For the decreasing quaintly case, we have
\begin{equation}\label{newHuex2}
\begin{cases}
\varphi_t+\varphi_xf(\varphi_x)=0\\
\varphi(0,x)=\int_{-\infty}^{x}\rho_0(\tau)\,\d\tau.
\end{cases}
\end{equation}
 In this case, we have a Hamilton--Jacobi equation with a concave Hamiltonian, $H(p)=pf(p)=-p^2+p$. Using the Hopf-Lax formula once
more, we explicitly solve \eqref{newHuex2}:
\begin{equation}\label{sol2}
\varphi(t,x)=\max\limits_{y}\left\lbrace tL\left( \frac{x-y}{t}\right) + \int_{-\infty}^{y}\rho_0(\tau)\,\d\tau\right\rbrace,
\end{equation}
where $L$ is Legendre transform of $H$ and the maximum in \eqref{sol1} is attained because $\varphi(0,\cdot)$ is Lipschitz and $L$ is concave.

Note that \eqref{sol1} also holds when $\rho_0$ is an increasing function, and \eqref{sol2} also holds when $\rho_0$ is a decreasing function.
\end{example}
\begin{remark} The arguments of Example \ref{ex} also hold for the congestion case; that is, we have \eqref{newHuex1} with $f(p)=\frac{k_1}{(k_2p)^\beta}$,
where $0<\beta<1/2$ and  $ k_1$, $k_2>0$.
\end{remark}

\def\cprime{$'$}

\end{document}